\documentclass{article}
\usepackage{graphicx} 
\usepackage{amsmath}
\usepackage{amssymb}
\usepackage{amsthm}
\usepackage{graphicx} 
\usepackage{a4wide}
\usepackage{hyperref}
\usepackage{soul}
\usepackage{mathtools}
\usepackage{tabulary}
\usepackage{booktabs}
\usepackage{mathrsfs}
\usepackage{amssymb}
\usepackage[dvipsnames]{xcolor}  
\usepackage{soul}
\usepackage{multirow}
\usepackage{mathtools} 
\let\oldforall\forall
\let\forall\undefined
\DeclareMathOperator{\forall}{\oldforall}
\usepackage{amsmath}
\usepackage{blindtext}
\usepackage{nccmath}
\usepackage{amsthm}
\usepackage{multicol}
\usepackage{subcaption}
\usepackage{float}
\usepackage{lscape}
\usepackage{xcolor,soul}
\usepackage[normalem]{ulem}
\usepackage{changepage}
\usepackage{algorithm}
\usepackage{comment}
\usepackage{float}
\usepackage{setspace}
\usepackage{placeins}
\usepackage[figuresright]{rotating}
\usepackage{graphicx}

\usepackage{lineno}
\usepackage{natbib}
\usepackage{makecell}
\usepackage{booktabs}
\usepackage{array}
\usepackage{bbm}
\usepackage{mathtools, nccmath}

\newcommand{\argmin}[1]{\underset{#1}{\mathrm{argmin}}}

\newcommand{\x}{\textbf{x}}

\renewcommand{\P}{\mathbb{P}}

\newcommand{\F}{\mathcal{F}}
\renewcommand{\l}{\mathbf{l}}
\renewcommand{\L}{\mathcal{L}}
\newcommand{\Lap}{\mathbf{L}}

\renewcommand{\r}{\mathbf{r}}

\newcommand{\V}{\mathcal{V}}
\newcommand{\R}{\mathbb{R}}
\newcommand{\Res}{\mathbf{R}}
\renewcommand{\d}{\mathrm{d}}

\newcommand{\one}{\mathbf{1}}
\renewcommand{\c}{\mathbf{c}}
\newcommand{\C}{\mathcal{C}}
\newcommand{\Lam}{\mathbf{\Lambda}}
\newcommand{\s}{\mathbf{s}}

\newcommand{\A}{\mathcal{A}}

\newcommand{\Q}{\mathbf{Q}}

\newcommand{\I}{\mathbf{I}}
\newcommand{\E}{\mathcal{E}}

\newcommand{\tran}{\mathrm{T}}

\newcommand{\Exp}{\mathbb{E}}

\newcommand{\set}[1]{\{#1\}}
\newtheorem{lemma}{Lemma}

\newtheorem{corollary}{Corollary}
\newtheorem{theorem}{Theorem}

\title{Deterministic and Random Bipartite Matching on General Networks: Convex Flow Reformulation, Asymptotic Properties, and Fast Algorithms}

\author{Yuhui Zhai, Yanfeng Ouyang}
\date{}
\begin{document}

\maketitle
\begin{abstract}
Minimum-distance bipartite matching on general networks has numerous applications in sciences, engineering, business, and sociology contexts. In transportation, the 
problem frequently arises in mobility service systems, such as 
ride-hailing, freight logistics, and crew assignment, where 
an equal number of supply and demand points are distributed along the edges of an arbitrary given network. 

This paper first focuses on deterministic problems. We present an exact edgewise-separable convex-flow reformulation which explores the cumulative supply-demand imbalance profile of each network edge and represents each matched point pair as one unit of flow routed through the network. 
By introducing a smooth monotone rearrangement approximation of the edge-wise imbalance profiles, which is shown to be asymptotically exact, 
the convex-flow reformulation's can be solved efficiently by standard algorithms such as the Newton's method. If we further conduct a first-order resistance-based approximation of the convex program
, a one-step Laplacian-based estimator for the matching solution and cost can be analytically derived in closed forms. 

Next the paper studies random problems where supply and demand point distributions are described probabilistically. We show that the expected optimal matching distance scales with the square root of the number of points if the supply/demand point distributions are identical, or linearly otherwise. 
In the former case, the optimal flow is proven to be centered, symmetric, and sub-Gaussian, such that the fraction of matching pairs across different edges vanishes asymptotically. 
In the latter case, the limiting resistance network characterizes how supply-demand imbalance is redistributed across the network and further provides an asymptotic closed-form cost estimate. This also motivates a fast algorithm that approximate the optimal flow by projecting the limiting flow onto the flow-conservation feasible set.

Numerical experiments with a variety of networks of varying sizes and topologies show that the proposed convex-flow and resistance-based estimators closely approximate the exact matching cost while substantially reducing computation time. In contrast to the original formulation, whose computational complexity scales cubically with the number of points, once the points are sorted on the edges, the computational complexity of all proposed estimators is independent of the number of points and instead depends only on the network size. 
The convex-flow formulation consistently achieves the highest accuracy, whereas the resistance-based and one-shot estimators provide increasingly faster alternatives, particularly for large-scale problems. 
Moreover the proven theoretical properties of the random matching solution are numerically verified by large-scale simulations. 
\end{abstract}

\section{Introduction}
Minimum-cost bipartite matching, also known as the linear assignment problem, is a fundamental problem in combinatorial optimization. Given two finite sets of equal cardinality and a cost associated with pairing two elements across the sets, the goal is to find a one-to-one matching that minimizes the total assignment cost. This problem has numerous applications in science and engineering fields, and it is a basic one in mobility and logistic systems as well. 
Examples include matching drivers to ride-hailing passengers \citep{OUYANG2023217, ozkan2020dynamic}, parcel couriers to delivery orders \citep{yildiz2019meal}, and optimal drone assignment in truck-drone hybrid delivery \citep{murray2015flying}. Matching decisions directly affect passenger waiting time, empty vehicle mileage, and service efficiency and reliability. In the field of transportation, supply and demand points are often located along the edges of a transportation network, and matching costs are induced by shortest-path distances on the network. Also, most transportation applications (e.g., ride hailing) involve two very large sets of points (e.g., on the order of millions), and hence 
efficiency in solving such problems 
is important for both real-time operations and system-level planning. 

There is a rich literature on ways to solve the bipartite matching problem efficiently. Classical methods include the Hungarian algorithm \citep{kuhn1955hungarian}, the shortest augmenting path algorithms \citep{jonker1987shortest}, and auction-based methods \citep{bertsekas1988auction}. Modern implementations of these algorithms are highly optimized and can solve general cost-matrix instances with strong cubic-time performance \citep{JV_algorithm}. 
These methods are broadly applicable, since they can accommodate arbitrary pairwise cost matrices without relying on any problem assumptions. 
However, this generality also prevents them  exploiting additional geometric or network structures that arises when matching cost is induced by spatial locations or transportation networks. 
In large-scale mobility applications, constructing the full pairwise cost matrix and solving the assignment problem can become computationally expensive, especially when matching decisions must be updated online \citep{abeywickrama2021optimizing}. Multiple approaches have been proposed to accelerate matching by pruning the candidate set \citep{yang2020optimizing}, constructing sparse bipartite graphs \citep{alonso2017ondemand}, batching nearby requests \citep{zheng2018order}, or avoiding unnecessary pairwise travel-cost computations \citep{abeywickrama2021optimizing}. These methods improve scalability, but they still operate within a pairwise candidate-graph or cost-matrix framework that scales polynomially with the number of points. 

A related perspective comes from the theory of optimal transport, which studies how to move mass from one probability distribution to another at minimum transportation cost \citep{santambrogio2015optimal}.  
Minimum-cost bipartite matching can be viewed as the finite empirical counterpart of the optimal transport problem: each supply point and demand point is an atom of equal mass, and the goal is to transport every supply atom to one demand atom at minimum total cost. 
The optimal-transport viewpoint is especially useful when the ground cost has spatial or network structure. On a line, the 1-Wasserstein distance equals the area between cumulative distribution curves, so the matching cost can be computed from an imbalance curve rather than from all pairwise distances \citep{santambrogio2015optimal}. A similar simplification is available on trees. Since any two points on a tree are connected by a unique path, cutting an edge uniquely determines the amount of mass that must cross the cut. The transport cost can therefore be written as the sum over all edge cuts, of the area under the absolute edge-level cumulative imbalance curve \citep{evans2012phylogenetic}. While these results are inspiring and promising, the problem in general networks are more challenging because the presence of cycles make the cross-edge flow no longer uniquely determined by local edge cuts. Along this direction of inquiry, \cite{treleaven2013roadmap} study bipartite matching with shortest-path distances on a general network and reduce the problem to a minimum-cost convex-flow formulation. Later, \cite{treleaven2014emdroadmap} consider the 
continuous distribution analogue and formulate the Earth Mover's Distance on 
road networks as a finite-dimensional convex optimization problem. 
These studies mainly focus on deterministic reformulation of matching problems on a network
, but do not explore properties of the problem as well as efficient solution methods. In addition, little attention has been cast on analyzing the stochastic behavior of the matching problem under randomly distributed supply and demand. 

In transportation applications, interest has also been cast on developing closed-form analytical formulas that explain how optimal matching cost depends on problem size, distance metric, spatial distributions, and/or network structure. Such results are useful for planning applications, e.g., to have a simple way to evaluate expected matching cost under any system planning decisions. In ride-hailing system, for example, spatial models are often used to approximate pickup time of matching distance between idle vehicles and waiting passengers \citep{wang2024ondemand, fan2026approximationmodelssharedmobility}. 
A stream of literature on this topic studies the expectation of the optimal matching cost when supply and demand points are described by random distributions, which is called the random bipartite matching problem. Existing studies establish some asymptotic scaling laws and predict the expected matching cost in continuous metric spaces \citep{mezard1985replicas, caracciolo2014scaling, Caracciolo_2017, shen2024expected, shen2026dynamic}. 
Studies in this direction, however, have not addressed the problem on a general network. Many related questions remain open, such as where matching inefficiency is generated, how network topology affects the matching performance, and how local supply-demand imbalances lead to matches across the network. 

In light of these gaps, this paper develops an interpretable approximation framework for random bipartite matching on general networks. We use edgewise cumulative supply-demand imbalance profiles as the representation of a realized matching instance, rather than point-based bipartite assignment. This greatly reduces the size of the optimization problem (from number of points to number of network edges). While existing convex-flow formulation 
shows that any matching problem instance can be solved through edgewise potentials 
\citep{treleaven2013roadmap, treleaven2014emdroadmap}, our focus is to use this representation as an analytical and computation tool. In particular, we study how the shape of each edge's imbalance profile determines the marginal cost of perturbing the boundary flow, which leads to an edge-induced resistance and a Laplacian-based approximation that estimates the total matching cost and the network-level matching flow. 

This paper makes four main contributions. First, we build on the edgewise convex-flow representation and introduce a smooth monotone-rearrangement approximation of the edge imbalance profiles. This approximation yields a smooth convex-flow formulation that can be solved accurately and efficiently using standard numerical methods, such as Newton’s method. Second, we derive KKT conditions of the smoothed convex-flow formulation and show that their first-order linearization admits an electric-network interpretation, where each edge has an induced resistance determined by the local slope. This resistance characterizes the edge's local sensitivity to flow perturbation and leads to a computationally efficient electric-network analogy and approximation. Third, we provide a series of analyses on the random bipartite matching problem on networks. We establish that the asymptotic scaling of the expected matching cost is of order $\sqrt{n}$ when supply and demand points follow the same distribution, or of order $n$ otherwise. 
We further characterize the corresponding limiting convex-flow problems when the number of points goes to infinity: when the supply and demand points follow identical distributions, we show that the likelihood of inter-edge matching diminishes as sub-Gaussian. 
Fourth, when the point distributions are non-identical, we can use the asymptotic limiting problem and the induced edge resistance as a basis to construct a one-shot estimator for deterministic large-scale random matching problems. The estimator adjusts the limiting flow through a single projection to satisfy the realized flow-conservation constraints, thereby avoiding iterative computations. It provides a computationally very efficient approximation when the supply and demand distributions differ and the problem size is sufficiently large.

The remainder of the paper is organized as follows. Section \ref{sec: convex flow} presents the deterministic convex-flow formulation, the smooth monotone-rearrangement approximation, the resistance-based linearization, and extensions to a few more general problems. Section \ref{sec: RBMP} analyzes random bipartite matching on general networks and derives properties of asymptotic limiting flow and cost approximations. Section \ref{sec: numerical} 
validates the proposed formulations with numerical experiments on different networks. Section \ref{sec: conclusion} concludes the paper and discuss future research directions. 

\section{Separable Convex Flow Formulation} 
\label{sec: convex flow}
This section develops a separable convex flow representation of the deterministic bipartite matching problem on a general network, which features 
an edgewise additive objective function that 
drastically simplifies both computation and asymptotic analysis. 

In what follows, unless stated otherwise, we exclusively use bold symbols to denote vectors or matrices, and assume that a vector is a column vector. 
We also use the convention that a vector or a matrix can be formed by properly stacking indexed scalars, and the vector with a subscript indicates the corresponding scalar element; for example, $\mathbf{x} := (x_a)_{a\in A}$ denotes the vector indexed by the elements in a set $A$, and $\mathbf{x}_a = x_a$ the element corresponding to $a\in A$.  

\subsection{Basic setting}
In a general undirected network $(\V, \E)$, edge $e\in \E$ has length $l_e\geq 0$ and end nodes $v^-_{e} \in \V$ and $v^+_{e} \in \V$. The edge vector is denoted as $\l = (l_e)_{e\in \E}$. 
To uniquely track the location of any point $w$ on edge $e$, which is $x_w$ distance away from $v^-_{e}$, we impose an arbitrary local coordinate axis from $v^-_{e}$ to $v^+_{e}$, such that this point will be recorded to have coordinates $(e, x_w) \in \E \times [0,l_e]$. 
The collection of points on edge $e\in \E$ that is within distance $x$ away from $v^-_{e}$ can be represented by the following set: 
\begin{align*}
\L_e(x) := \{(e,x_w) : x_w \in [0,x]\}.
\end{align*}
Clearly, all points on edge $e$ are represented by the set $\L_e(l_e)$, and all points in the entire network are in the following set:
\begin{align*}
\L := \bigcup_{e\in \E} \L_e(l_e), 
\end{align*}
and $L = \sum_{e\in \E}l_e$ is the total edge length of the network. 
Accordingly, based on the local coordinate axis of each node, the sets of outflow and inflow arcs 
can be defined respectively as follows:
\begin{align*}
\delta^{-}(v) := \{ e \in \E : v_e^- = v \},\quad
\delta^{+}(v) := \{ e \in \E : v_e^+ = v \}, \quad \forall v \in \V.
\end{align*}

Now, suppose a finite sets of supply points $U$ and a finite set of demand points $V$ are distributed along the edges of the networks; i.e., $U \subset \L, V \subset \L$. These two sets are of equal cardinality $n$; i.e., $|U| = |V| = n$. The distance between any two points is measured along the shortest network path, and denoted by $d(\cdot, \cdot)$, while ties, if any, are broken arbitrarily. 
For notation convenience, we further specify the point subsets on each edge; i.e., 
\begin{align*}
U_e := U \cap \L_e(l_e), 
\quad
V_e := V \cap \L_e(l_e), \quad \forall e \in \E. 
\end{align*}
We seek the minimum-distance bipartite matching among these two sets of points, which can be formulated as the following standard integer programming (IP):
\begin{equation}
\label{eq: IP}
\begin{aligned}
\min_{\{y_{uv}\}} \quad & Z^{\text{IP}} = 
\sum_{u\in U}
\sum_{v\in V} y_{uv}d(u,v),
\\
\text{s.t.} \quad &
\sum_{u\in U} y_{uv} = 1, \quad \forall v\in V,\\&
\sum_{v\in V} y_{uv} = 1, \quad \forall u\in U,\\
&y_{uv} \in \{0, 1\}, \quad \forall u\in U,\forall v\in V. 
\end{aligned}    
\end{equation}
Such a formulation enjoys total unimodularity and hence can be efficiently solved in polynomial time; e.g., the state-of-art Jonker-Volgenant algorithm \citep{JV_algorithm} has a time complexity of $O(n^3)$. 

\subsection{Flow-based formulations}
Given any bipartite matching solution, we send one unit of flow between every matched pair $(u,v)\in U\times V$ along its shortest path from $u \in U$ to $v \in V$. This flow is considered positive  (with a value of $+1$) at a location $(e,x)\in \L_e(l_e)$ if its shortest path traverses this location in the direction from $v_e^-\to v_e^+$, or negative (with a value of $-1$) otherwise. The total net flow traversing $(e,x)$ from all matched pairs is denoted $f_e(x)$. Its value varies within $e$ based on the distributions of supply and demand points inside this edge, which can be described by a local cumulative net supply-demand \textit{imbalance profile} on $e$ up to any position $x$, as follows  
\begin{align*}
S_e(x)
&:= |U\cap \L_e(x)| - |V \cap \L_e(x)|,
\quad \forall e \in \E, x \in [0,l_e]. 
\end{align*}
While the matched points could be on different edges, the values of $f_e(x)$ and $S_e(x)$ are clearly related, as summarized in the following lemma. 
\begin{lemma}
\label{lemma: fx_eq_f0_plus_Sx}
For any feasible matching solution in the network,  
\begin{align*}
f_e(x) = f_e(0) + S_e(x), \quad \forall e \in \E, x\in [0, l_e].
\end{align*}
\end{lemma}

\begin{proof}
While this lemma should be intuitive, we still give a formal proof here. 
First, along any edge $e$, in between any two consecutive points in $U_e\cup V_e$, the value $f_e(x)$ must stay constant because no supply or demand point exist inside this interval (and hence no additional flow is generated). 
Next, we show that $f_e(x)$ increases (or decreases) by one unit at each supply (or demand) point.
For a point $w\in U_e$ located at $x_w$, if we choose $\varepsilon \rightarrow 0^+$ such that no other point from $U_e\cup V_e$ lies in $(x_w-\varepsilon, x_w+\varepsilon)$, then we can claim that $f_e(x_w+\varepsilon) = f_e(x_w-\varepsilon) + 1$. To see this, let $w'$ be its matched point on the same edge or elsewhere. 
    If flow $w\to w'$ is in direction $v_e^-\to v_e^+$, it does not traverse any $x\in (x_w-\varepsilon, x_w)$, 
    but does traverse all $x\in (x_w, x_w+\varepsilon)$. All other matched pair flows contribute to $f_e(x_w+\varepsilon)$ and $f_e(x_w-\varepsilon)$ equally. Hence, $f_e(x_w+\varepsilon) = f_e(x_w-\varepsilon)+1$.
    If $w\to w'$ is in direction $v_e^+\to v_e^-$, it does not traverse any $x\in (x_w, x_w+\varepsilon)$ 
    but does traverse all $x\in (x_w-\varepsilon, x_w)$. 
    Hence, $f_e(x_w-\varepsilon) = f_e(x_w+\varepsilon)-1$.
In case the point at $x_w$ belongs to the other set, $w\in V_e$, by symmetry, it can be proven similarly that $f_e(x_w+\varepsilon) = f_e(x_w-\varepsilon) - 1$. 
This implies that, along edge $e$ from $v_e^-$ to $v_e^+$, the variation of net flow $f_e(x)$ 
is captured exactly by the cumulative imbalance function $S_e(x)$. This completes the proof. 
\end{proof}

Lemma \ref{lemma: fx_eq_f0_plus_Sx} shows that, for any feasible matching solution, the flow profile on each edge is completely determined by the set of \textit{boundary flow }values $\{f_e(0)\}$ together with the local imbalance profiles $\{S_e(x)\}$ (which is already known). 
This motivates us to define the following potential function ${\Psi}_e(\cdot; \cdot): \mathbb{R}\to\mathbb{R}$ induced by any boundary-flow value $c_e$ and any edge profile $\xi_e:[0, l_e] \to \R$:
\begin{align*}
\Psi_e(c_e; \xi_e) : = \int_0^{l_e} |\xi_e(x)-c_e|\d x. 
\end{align*}
The following lemmas present the basic properties of such a potential function, which is convex and Lipshitz continuous.  

\begin{lemma} 
\label{lemma: Psi_e_Lipchitz}
For any given $\xi_e$, $e\in \E$, the function $c\mapsto\Psi_e(c;\xi_e)$ is convex. For any $c_1, c_2\in \R$, and any edge profiles $\xi_1, \xi_2 \in L^1([0, 1])$, $\Psi_e$ has the Lipschitz property
\begin{align*}
|\Psi_e(c_1; \xi_1) -  
\Psi_e(c_2; \xi_2)| \leq 
l_e\left(\|\xi_1 - \xi_2\|_{L^1} + |c_1 - c_2|\right).
\end{align*} 
When $\xi_e $ takes integer values only, it is also affine on each open interval $c \in (k, k+1)$, $\forall k\in \mathbb{Z}$. 
\end{lemma}
\begin{proof}
We prove these properties by definition. For all $c_1, c_2\in \mathbb{R}$, and $\theta \in [0, 1]$, by triangular inequality, we prove convexity from the following:
\begin{align*}
&\theta{\Psi}_e(c_1;\xi_e) + (1-\theta){\Psi}(c_2; \xi_e)
= \int_0^{l_e} |\theta(\xi_e(t)-c_1)|\d t + \int_0^{l_e} |(1-\theta)(\xi_e(t)-c_2)|\d t\\
&\geq \int_0^{l_e} |\theta(\xi_e(t)-c_1)+(1-\theta)(\xi_e(t)-c_2)|\d t
=\Psi(\theta c_1 + (1 - \theta)c_2; \xi_e).
\end{align*}
Meanwhile, Lipschitz continuity is shown as follows:
\begin{align*}
|\Psi_e(c_1; \xi_1)-{\Psi}_e(c_2; \xi_2)| = \bigg|\int_0^{l_e} |\xi_1(x) - c_1| - |\xi_2(x) - c_2|\d x\bigg| 
\leq 
l_e\left(\|\xi_1 - \xi_2\|_{L^1} + |c_1 - c_2|\right).
\end{align*}
Finally, when $\xi_e $ takes integer values only, consider an arbitrary $c\in (k,k+1)$, where $k\in \mathbb{Z}$.  
\begin{align*}
|\xi_e(x) - c| = 
\begin{cases}
\xi_e(x) - c, & \xi_e(x) \geq k+1,\\
c - \xi_e(x), & \xi_e(x) \leq k.       
\end{cases}
\end{align*}
Thus, by definition, 
\begin{align*}
\Psi_e(c; \xi_e) &= \int_{0}^{l_e} |\xi_e(x) - c|\d x 
=\int_{\xi_e(x) \geq k+1} (\xi_e(x) - c )\d x 
+ \int_{\xi_e(x) \leq k} (c - \xi_e(x))\d x \\
&= \left[\int_{x:\xi_e(x) \leq k}\d x -\int_{x:\xi_e(x) \geq k+1}\d x \right] c
+ \left[\int_{x:\xi_e(x) \geq k+1} \xi_e(x)\d x - \int_{x:\xi_e(x) \leq k} \xi_e(x)\d x\right],
\end{align*}
which is clearly an affine function with respect to $c$. 
This completes the proof. 
\end{proof}



When $c_e := -f_e(0)$ and $\xi_e = S_e$, the potential function $\Psi_e(c_e; \xi_e)=\int_0^{l_e} |f_e(x)|\d x$, which gives the total absolute distance traversed by the net flow on edge $e$. 
Let $\c = (c_{e})_{e\in\E}$ denote the vector of boundary flows.  
Note we already know the following input data: (i) 
the vector of edge-wise total imbalances $\mathbf{{s}} = (S_e(l_e))_{e\in \E} \in \mathbb{Z}^{|\E|}$; and 
(ii) the 
outflow (or inflow) node-edge incidence matrix $\I^+ $ (or $ \I^-$), whose $(v,e)$-th entry equals 1 if $e\in \delta^+(v)$ (or if $e\in \delta^-(v)$), or 0 otherwise. 
Any feasible flow $\{f_e(x)\}$ must satisfy the conservation ---
at each nodal junction, the net in-coming flow must be equal to the net out-going flow; i.e., for each $v\in \V$:
\begin{align*}
&\sum_{ e\in \delta^-(v)}
f_e(0) - \sum_{e\in \delta^+(v)}
f_e(l_e) = 
\sum_{e\in \delta^-(v)}
f_e(0) - \sum_{e\in \delta^+(v)}
\left[f_e(0) + S_e(l_e)\right] = 0.
\end{align*}
Since $c_e:= -f_e(0)$, the above can be written into the following matrix form:
\begin{align} 
\label{conserv}
\I\c = \I^+\s,
\end{align}
where $\I = \I^+ - \I^-$. We will next show that the optimal objective value of \eqref{eq: IP} can be obtained by solving the following formulation:
\begin{equation}
\begin{aligned} 
\label{eq: CF}
\min_{\mathbf{c}\in \mathbb{Z}^{|\E|}} \quad Z^{\mathrm{CF}} = \sum_{e\in \E} \Psi_e(c_e; S_e), \quad
\text{s.t.} \quad \eqref{conserv}.
\end{aligned}
\end{equation}
In other words, we will show $Z^\mathrm{CF} = Z^\mathrm{IP}$. The proof proceeds in two steps. First, Lemma \ref{lemma: fe*_gives_feasible_solution} below shows that every feasible boundary flow vector $\mathbf{c}$ can be converted into a feasible matching solution with the same objective value, and hence $Z^\mathrm{IP}\leq Z^\mathrm{CF}$. Next, Lemma \ref{lemma: fe*_lower_bound} shows that every feasible matching induces a feasible boundary flow vector whose objective value is no larger than the corresponding matching distance, and hence $Z^{\mathrm{CF}} \leq Z^{\mathrm{IP}}$.  


\begin{lemma}
\label{lemma: fe*_gives_feasible_solution}
Given any feasible flow profile $\{f_e(x), \forall (e,x)\}$, after sorting all points on each network edge based on the local $x$-coordinate (in $O(n\log n)$ time), there exists an $O(n)$-time  algorithm to construct a feasible solution to \eqref{eq: IP} with an objective value of 
\begin{align*}
\sum_{e\in \E}\int_0^{l_e} |f_e(x)|\d x.
\end{align*}
\end{lemma}
\begin{proof}
A feasible matching solution can be constructed by slicing the area between the $f_e(x)$ curve and $x$-axis into horizontal unit-height layers (i.e., height $=1$). Since $f_e(x)$ takes an integer value everywhere, each horizontal layer corresponds to two consecutive integers, $k$ and $k+1$, where $k\in \mathbb{Z}$.  
For $k\geq 0$, the layer begins at an upward jump of $f_e(x)$ from $k$ to $k+1$, which can be caused by either a supply point at that location or a unit flow entering the edge (which has to be at $x=0$). The layer ends at a downward jump from $k+1$ to $k$, which can be caused by either a demand point at that location or a unit flow leaving the edge (which has to be at $x = l_e$). 
Thus, each upward jump from $k$ to $k+1$ can be paired with the first subsequent downward jump from $k+1$ to $k$, and one unit of flow is sent along the interval between these two jump locations. This interval length (or the area size of the unit-height layer) corresponds to either a local match distance between a supply point and a demand point on edge $e$, or a local segment of a shortest path that traverses $e$.
For $k < 0$, everything is exactly similar and hence omitted. 
This is illustrated in Figure \ref{fig: cost_eq_int_abs_fx}(a), where 
supply and demand points are represented by red circle markers and blue triangle markers, respectively. Each arrow indicates a point-to-point match. 
Summing over the areas of all horizontal unit-height layers gives exactly 
$\sum_{e\in \E} \int_0^{l_e} |f_e(x)|\d x$.

The above slicing idea also leads to the following scanning procedure which, in linear time, can construct a feasible matching solution in a network. 
First, virtual supply or demand points are added to the two endpoints of each edge $e$ wherever the boundary flows are non-zero; i.e., 
if $f_e(0^+) > 0$ (or $<0$), we place $|f_e(0^+)|$ virtual supply (or demand) points at $x = 0$; 
if $f_e(l_e^-) > 0$ (or $<0$), we place $|f_e(l_e^-)|$ virtual demand (or supply) points at $x = l_e$. 
Then we scan all real and virtual points on each edge in an increasing order of their local $x$-coordinates, and use the above slicing procedure to form matches. It is easy to show that this procedure can be done in $O(n)$ computation time with the use of a last-in first-out (LIFO) stack, since 
each supply or demand point corresponds to exactly two operations (addition to or removal from the stack). 

After all edges are processed, any path segments that involve virtual points must be concatenated 
to obtain supply-demand point matches. Due to flow conservation, the virtual supply and demand points as every network vertex $v \in \V$ must be paired. 
We can arbitrarily concatenate two path segments (along two edges) that involve one virtual point of each type at the same vertex. 
This operation will be repeated until no virtual point remains in the path segments, and a feasible matching solution is found. This concatenation procedure does not affect the total matching distance, and hence a feasible matching solution 
%
to \eqref{eq: IP} has an objective value of $\sum_{e\in \E} \int_0^{l_e} |f_e(x)|\d x$. This completes the proof.
\end{proof}
 
\begin{lemma}
\label{lemma: fe*_lower_bound}
For any feasible matching solution that induces the set of flow profiles $\{f_e(x), \forall (e,x)\}$,  the total matching distance is bounded from below by
$\sum_{e\in \E} \int_0^{l_e}|f_e(x)|\d x$.
\end{lemma}

\begin{proof}
For any edge $e\in \E$ and position $x\in [0, l_e]$,
by definition, the net flow $f_e(x)$ is the number of shortest paths crossing $x$ in the direction $v_e^-\to v_e^+$ minus the number of those crossing $x$ in the opposite direction $v_e^+\to v_e^-$.
Thus, the total number of shortest paths crossing $x$, given by the sum of these two counts, is at least $|f_e(x)|$. Now consider an infinitesimal edge segment $[x, x+\d x]$.  Any unit of flow that crosses this infinitesimal segment contributes $\d x$ to the total matching distance. 
Hence, the total contributions from all flow paths crossing this infinitesimal segment is at least $|f_e(x)|\d x$. This is illustrated in Figure \ref{fig: cost_eq_int_abs_fx}(b). Since this argument holds for every $x\in [0, l_e]$ and for all $e\in \E$, integrating this quantity over $x$ 
and summing it across $e$,  
$\sum_{e\in \E} \int_{0}^{l_e}|f_e(x)| \d x$, 
gives a lower bound to the total matching distance.
This completes the proof. 
\end{proof}

\begin{figure}[ht]
    \centering
    \includegraphics[width=1\linewidth]{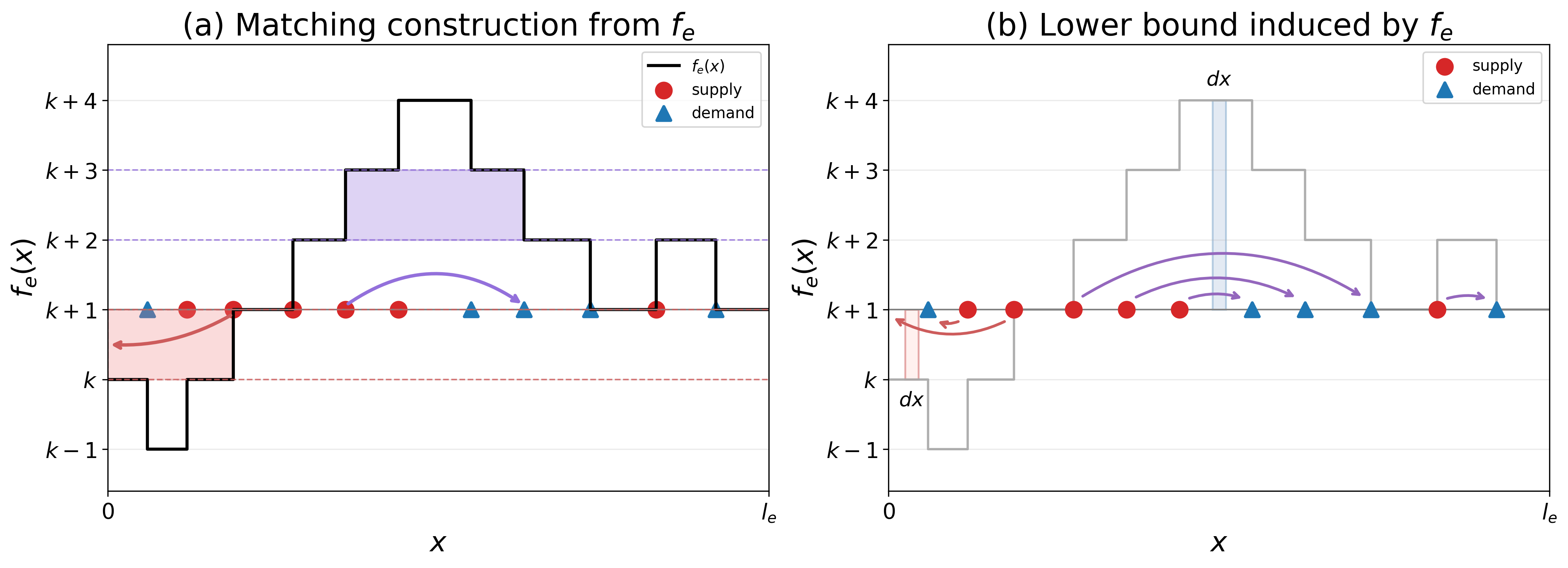}
    \caption{Illustrations for the proofs of Lemmas \ref{lemma: fe*_gives_feasible_solution} and \ref{lemma: fe*_lower_bound}.}
    \label{fig: cost_eq_int_abs_fx}
\end{figure}

Lemmas \ref{lemma: fe*_gives_feasible_solution}
and \ref{lemma: fe*_lower_bound} together establish the equivalence between the optimal objective values from \eqref{eq: IP} and \eqref{eq: CF}, as summarized in the following theorem. 

\begin{theorem}
\label{theorem: CF}
At optimality, $Z^\mathrm{IP} = Z^\mathrm{CF}$. 
\end{theorem}

In the above analysis, the boundary flow $c_e=-f_e(0)$ 
takes integer values for all $e \in \E$.
Yet, we notice that the flow conservation constraints in \eqref{eq: CF} 
satisfy total unimodularity. As such, it may be convenient to relax the integrality requirement and allow $c_e$ to take real values, as in the following formulation. 
\begin{equation}
\begin{aligned} 
\label{eq: RLX}
\min_{\mathbf{c}\in \mathbb{R}^{|\E|}} \quad & Z^\mathrm{RLX} = \sum_{e\in \E} \Psi_e(\c_e; S_e), \quad
\text{s.t.} \quad \eqref{conserv}. 
\end{aligned}    
\end{equation}    

The relaxed optimal objective $Z^\mathrm{RLX}$ can be shown to also equal that of the original formulation \eqref{eq: CF}, as stated below.

\begin{theorem}   
\label{theorem: RLX}
At optimality, 
%
$Z^\mathrm{CF} = Z^\mathrm{RLX}$. 

\begin{proof}
Let $\c^\mathrm{rlx} \in \mathbb{R}^{|\E|} $ denote the optimal solution to \eqref{eq: RLX}. 
For any given integer-valued $S_e, e\in \E$, 
Lemma \ref{lemma: Psi_e_Lipchitz} says that $\Psi_e(c_e; S_e)$ is 
affine over  $c_e \in [\lfloor\c^\mathrm{rlx}_e\rfloor, \lceil\c^\mathrm{rlx}_e\rceil]$.  
Since $\I$ is totally unimodular and $\I^+\s$ is integral, there exists an integral extreme point $\c^{\mathrm{int}} \in \mathbb{Z}^{|\E|}$ that is an optimal solution to \eqref{eq: RLX} while satisfying the following additional constraints: 
$\lfloor\c^\mathrm{rlx}_e\rfloor\leq \c_e^\mathrm{int} \leq \lceil\c^\mathrm{rlx}_e\rceil, \forall e\in \E$.
The objective value of such a restricted problem is equal to $Z^\mathrm{RLX}$. Since $\c^\mathrm{int}$ is a feasible solution to \eqref{eq: CF}, we have $Z^\mathrm{CF} \leq Z^\mathrm{RLX}$. On the other hand, $Z^\mathrm{RLX}\leq Z^\mathrm{CF}$ since \eqref{eq: RLX} is a relaxed problem. This proves that $Z^\mathrm{RLX} = Z^\mathrm{CF}$. 
\end{proof}
\end{theorem}

Theorems \ref{theorem: CF} and \ref{theorem: RLX} reveal that the original pairwise matching problem \eqref{eq: IP} can be equivalently represented by a continuous flow formulation \eqref{eq: RLX}. Although this formulation can already be solved by standard optimization techniques, it still does not reveal interesting properties mainly because the main input data, imbalance profiles $\{S_e(x), \forall (e,x)\}$ are not differentiable everywhere. 
In the following section, we 
further analyze an approximate formulation based on a smooth representation of $\{S_e(x), \forall (e,x)\}$, which 
leads to a set of interpretable optimality 
conditions, and will be useful in analyzing the stochastic version of the bipartite matching problem.

\subsection{Approximate flow formulation and linear-time algorithm}
\label{sec: approximate flow and linear-time}

\begin{figure}[ht]
\centering\includegraphics[width=0.95\linewidth]{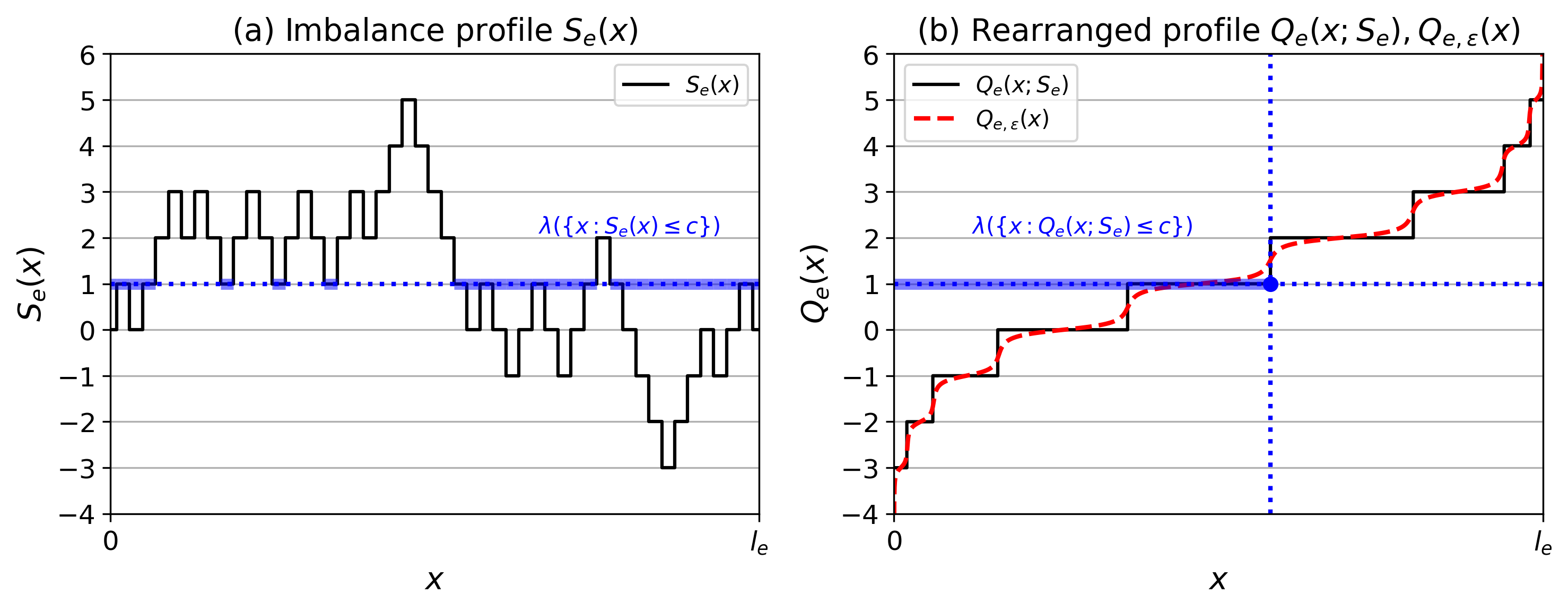}
\caption{Illustration of profiles $S_e, Q_e, Q_{e, \epsilon}$. }
\label{fig: vis_S_Q}
\end{figure}

To characterize and approximate $\{S_e(x), \forall (e,x)\}$, we introduce two additional functions that are associated with an edge profile $\xi_e$:  
\begin{align*}
&F_e(c; \xi_e) := \int_{x:\xi_e(x)\leq c}\d x, \quad
Q_e(t; \xi_e) := 
\inf \set{c: F_e(c; \xi_e)\geq t}.
\end{align*}
The first function $F_e:\mathbb{R}\to [0, l_e]$ gives the total length of 
edge $e$ where  $\xi_e(x) \leq c$.  It is easy to see that $F_e$ is non-decreasing and right-continuous. 
Since $F_e$ may not be differentiable, we use $F_e^-(c; \xi_e) := \lim_{y\uparrow c} F_e(y; \xi_e)$. Note that $Q_e$ is the generalized inverse of $F_e$, and geometrically, $Q_e(\cdot; \xi_e)$ gives the monotonic 
``cumulative distribution'' of the values in profile $\xi_e$. Since $\Psi_e$ depends only on the distribution of the profile values, i.e., the total edge length below each level $c_e$, replacing $\xi_e$ by $Q_e(\cdot; \xi_e)$ preserves the equivalent matching costs. For example, Figures \ref{fig: vis_S_Q}(a) and (b) visualize the distributional equivalence of $S_e$ and $Q_e(\cdot; S_e)$, 
because for any fixed $c$ value, the blue spans in the two figures have the same total length. 
Such a property is formalized in the following lemma: 

\begin{lemma}
\label{lemma: Se_eq_Qe}
For any $\xi_e$ and any $c\in \R$, 
\begin{align*}
\Psi_e(c; \xi_e) = \Psi_e(c; Q_e(\cdot; \xi_e)).
\end{align*}
\end{lemma}

\begin{proof}
By the generalized-inverse identity, for any $c\in \R$, 
\begin{align*}
Q_e(x; \xi_e)\leq c 
\quad \Leftrightarrow \quad
F_e(c; \xi_e) \geq x.
\end{align*}
Therefore,
$\lambda(\{x: Q_e(x;\xi_e)\leq c\}) = F_e(c; \xi_e)$, where $\lambda(\cdot)$ is the Lebesgue measure. 
By definition, 
$\lambda(\{x: S_e(x)\leq c\}) = F_e(c; S_e)$.
Hence, $\xi_e$ and $Q_e(\cdot; \xi_e)$ have the same total edge length under every level $c$. Since $F_e$ is non-decreasing, $Q_e(\cdot; \xi_e)$ is also non-decreasing. Therefore, $Q_e(\cdot; \xi_e)$ is a monotone rearrangement of $\xi_e$. Finally, the integrals over $\xi_e, Q_e$ are the same; i.e., 
\begin{align*}
\Psi_e(c; S_e) 
= \int_{0}^{l_e} |\xi_e(x) - c|\d x 
= \int_{0}^{l_e} |Q_e(x; \xi_e) - c|\d x  
= \Psi_e(c; Q_e(\cdot; \xi_e)).
\end{align*}
This completes the proof. 
\end{proof}

Lemma \ref{lemma: Se_eq_Qe} indicates that an equivalent formulation can be obtained by replacing the original profile $S_e$ by its monotone representation $Q_e(\cdot; S_e)$. This has two benefits. First, both the objective value and the optimality condition are preserved. 
Second, the rearranged representation $Q_e(\cdot; S_e)$ allows us to conveniently impose a smooth approximation, 
$Q_{e, \epsilon}$ controlled by a smoothing parameter $\epsilon > 0$. 
In this study, we assume that $Q_{e, \epsilon}$ is twice differentiable, strictly increasing, and that $Q_{e, \epsilon}$ converges to $Q_e(\cdot; S_e)$ in $L^1(0, l_e)$ as $\epsilon \downarrow 0$; i.e., 
\begin{align}
\label{eq: max_sup_S_minus_barS_leq_C}
Q_{e, \epsilon}\in \mathcal{C}^2([0, l_e]), \quad
Q_{e, \epsilon} > 0,
\quad
\lim_{\epsilon \downarrow 0} \max_{e\in \E} \|Q_{e, \epsilon}(\cdot) - Q_e(\cdot; S_e)\|_{L^1(0, l_e)} = 0.
\end{align} 
The red dashed curve in Figure \ref{fig: vis_S_Q}(b) shows an example of $Q_{e, \epsilon}$.

Now we propose the following approximate convex-flow formulation.
\begin{theorem}
\label{theorem: Z_ZCF_error_bounded}
The optimal total matching distance in a general network is approximated by the optimal objective value $Z$ of the following approximate formulation with respect to $\c$: 
\begin{equation}
\begin{aligned} 
\label{eq: smooth}
\min_{\mathbf{c}\in \mathbb{R}^{|\E|}} \quad & Z = \sum_{e\in \E} \Psi_e(c_e; Q_{e, \epsilon}), \quad
\text{s.t.} \quad \eqref{conserv}.
\end{aligned}    
\end{equation}    
with convergence as $\epsilon \downarrow 0$, 
\begin{align*}
\lim_{\epsilon \downarrow 0}|Z - Z^{\mathrm{CF}}| = 0.  
\end{align*}
\end{theorem}

\begin{proof}
We prove the theorem by showing the following:
\begin{align*}
&|Z-Z^{\mathrm{RLX}}|
= 
\left|\min_{\c} \sum_{e\in \E}\Psi_e(\c_e; Q_e(\cdot; S_e)) - \min_{\c}\sum_{e\in \E}\Psi_e(\c_e; Q_{e, \epsilon})\right| \\
&\le \min_{\c} 
\left|\sum_{e\in \E}\Psi_e(\c_e; Q_e(\cdot; S_e)) - \Psi_e(\c_e; Q_{e, \epsilon})\right| 
\leq  
\sum_{e\in \E}\int_{0}^{l_e} |Q_e(x; S_e) - Q_{e, \epsilon}(x)| \d x \to 0. 
\end{align*}
The first 
inequality follows the monotonicity of the minimum; the second inequality also follows Lemma \ref{lemma: Se_eq_Qe}, and the last convergence follows the assumptions in \eqref{eq: max_sup_S_minus_barS_leq_C}. 
From Theorem \ref{theorem: CF}, we know $Z^\mathrm{CF} = Z^\mathrm{RLX}$. 
This completes the proof. 
\end{proof}

The solution to such constrained optimization formulation can be easily obtained via analysis of its Lagrangian, as explained below. 
\begin{lemma}
\label{lemma: Psi_e_subdifferential}
For any fixed $\xi_e$ and any $c\in \R$, the subdifferential of $\Psi_e(c; \xi_e), \forall e\in \E$, is given by $\partial\Psi_e(c; \xi_e) = [2{F}_e^-(c; \xi_e)-l_e, 2{F}_e(c;\xi_e) - l_e]$.
\end{lemma}

\begin{proof}
The subdifferential of $|\xi_e(x)-c|$ is simply
\begin{align*}
\partial |\xi_e(x)-c|= 
\begin{cases}
\set{-1}, & \quad \xi_e(x)> c,\\ 
[-1, 1], &\quad \xi_e(x)= c,\\ 
\set{1} &\quad \xi_e(x) < c.
\end{cases}
\end{align*}
Thus, the subdifferential of ${\Psi}_e(c; \xi_e)$ is:
\begin{align*}
\partial {\Psi}_e(c; \xi_e) &= 
\int_0^{l_e} \partial |\xi_e(x)-c|\d x\\
&= \int_{x:\xi_e(x)<c} 1\d x + \int_{x:\xi_e(x)=c} \partial|\xi_e(t)-c| \d x + \int_{x:\xi_e(x)>c} - 1 \d x.
\end{align*}
Since 
$-\int_{x:\xi_e(x)=c} \d x \leq \int_{x:\xi_e(x)=c} \partial |\xi_e(x)-c| \d x \leq \int_{x:\xi_e(x)=c} \d x, 
$ 
we have:
\begin{align*}
\partial \Psi_e(c; \xi_e) 
&= \int_{x:\xi_e(x)<c} \d x - \int_{x:\xi_e(x) >c} \d x
+ \left[-\int_{x:\xi_e(x)=c} \d x, \int_{x:\xi_e(x)=c} \d x\right]\\
&= \int_{x:\xi_e(x)<c} \d x - \left(l_e- \int_{x:\xi_e(x)< c} \d x - \int_{x:\xi_e(x)= c} \d x\right)
+ \left[-\int_{x:\xi_e(x)=c} \d x, \int_{x:\xi_e(x)=c} \d x\right]\\
&= \left[2\int_{x:\xi_e(x)<c} \d x-l_e, 2\int_{x:\xi_e(x)\leq c} \d x - l_e\right] = [2F_e^-(c;\xi_e)-l_e, 2F_e(c;\xi_e)-l_e].
\end{align*}
This completes the proof. 
\end{proof}

We denote $\c^*$ as the optimal solution to \eqref{eq: smooth}. The optimality condition is summarized below. 

\begin{theorem}
\label{theorem: KKT}
An optimal solution $\c^*$ and its corresponding Lagrangian multiplier $\Lambda^* \in \R^{|\V|}$ must satisfy:
\begin{align*}
\c^* = \Q(\Lam^*) \quad \text{and}\quad 0 = \I\Q(\Lam^*) - \mathbf{I^+{{s}}},
\end{align*}
where 
\begin{align*}
\Q(\Lam) = 
\left(Q_{e, \epsilon}\left(\frac{l_e-(\I^\tran\Lam)_e}{2}\right)\right)_{e\in \E}.
\end{align*}
\end{theorem}

\begin{proof}
The edge-wise first-order conditions of the optimization problem are given by: 
\begin{align*}
&0\in \partial \Psi_e(\c_e^*; Q_{e, \epsilon}) + (\I^\tran\Lam^*)_e = [2{F}_e^-(\c_e^*; {Q}_{e, \epsilon}), 2{F}_e(\c_e^*; Q_{e, \epsilon})] - l_e + (\I^\tran\Lam^*)_e, \forall e. 
\end{align*}
Since $Q_{e, \epsilon}\in \mathcal{C}^2$ and $Q'_{e, \epsilon} > 0$, $F_e(\cdot; Q_{e, \epsilon})$ is continuous. 
This simplifies the optimality conditions to 
\begin{align}
\label{eq: KKT}
\frac{l_e-(\I^\tran\Lam^*)_e}{2} = F_e(\c_e^*; Q_{e, \epsilon})
\quad \Rightarrow \quad
\c_e^* = {Q}_e\left(\frac{l_e-(\I^\tran\Lam^*)_e}{2}; Q_{e, \epsilon}\right) = Q_{e, \epsilon}\left(\frac{l_e-(\I^\tran\Lam^*)_e}{2}\right), \forall e.
\end{align}
This proves the form of $\mathbf{c}^*$. Meanwhile, the rest of the theorem simply comes from the primal feasibility condition. 
This completes the proof. 
\end{proof}

The optimal solution $\c^*$ can be efficiently solved by a standard Newton method. Define the residual of the flow conservation as:  
\begin{align*}
\I\Q(\Lam) - \I^+\s.
\end{align*}
Its corresponding Jacobian with respect to $\Lam$ is: 
\begin{align*}
\nabla 
\left(\I{\Q}(\Lam)\right) 
= \I \nabla {\Q}(\Lam)
= -\frac{1}{2}\I\ \mathrm{diag}\left(\left({Q}'_{e, \epsilon}\left(\frac{l_e- (\I^\tran \Lam)_e}{2}\right)\right)_{e\in \E}\right)\ \I^\tran,
\end{align*}
where the diagonal term represents an edge ``conductance" matrix.\footnote{More about this matrix will be discussed in the next section.} 
In a general network with points sorted on each edge, each Newton iteration depends only on the network topology (and is independent of $n$). As such, the approximate convex-flow formulation \eqref{eq: smooth} can be solved in linear time complexity $O(n)$.  
Once $\c^*$ is obtained, we can construct a feasible matching solution to \eqref{eq: CF} based on the following two-step heuristic algorithm. 
\begin{enumerate}
    \item[Step 1] We define a new vector $\c^{\mathrm{rnd}}$ by rounding each element of $\c^*$ to the nearest integer, such that $|\c_e^{\mathrm{rnd}} - \c^*_e| \leq \frac{1}{2}, \forall e$. If $\c^\mathrm{rnd}$ satisfies the flow conservation constraints, then it is already a feasible solution to \eqref{eq: CF};
    \item[Step 2] If $\c^\mathrm{rnd}$ does not satisfy flow conservation, we heuristically perturb it into one, denoted $\c^{\mathrm{fes}}$, that satisfies both integrality and flow conservation. In so doing, we initially identify from $\c^\mathrm{rnd}$ two node subsets: $\V^+ = \set{v: \Delta_v > 0, v\in \V}$ and $\V^- = \set{v: \Delta_v < 0, v\in \V}$, and iteratively update them. In each iteration, we arbitrary pick two nodes, $u\in \V^-$ and $v\in \V^+$, and send a unit of flow along an arbitrary $u-v$ path. For each edge $e$ within the path, the value of $\c^{\mathrm{rnd}}_e$ increases by one unit if the path direction matches the edge direction; or otherwise it decreases by one unit. This operation does not affect flow conservation at any nodes other than $u, v$, and hence it decreases both the total deficit and total surplus values by one unit each. Repeat this operation $\frac{1}{2}\|\Delta\|_1$ times, we shall get $\Delta = 0$ and an integral feasible solution $\c^\mathrm{fes}$. 
\end{enumerate}
If points are already sorted on each edge, the above heuristic has a desirable $O(n)$ time complexity. The resulted feasible solution $\c^{\mathrm{fes}}$ has a bounded optimality gap, as shown in the theorem below. 

\begin{theorem}
\label{theorem: c*_to_IP}
Suppose $\c^\mathrm{int}$ is the optimal solution to \eqref{eq: CF}. Then 
\begin{align*}
0 \leq 
\sum_{e\in \E} \Psi_e(\c^\mathrm{fes}; S_e) - \sum_{e\in \E}\Psi_e(\c^\mathrm{int}; S_e) 
\leq
\left(\frac{|\E|}{2} + \frac{1}{2}\right) L + o(1).
\end{align*}
\end{theorem}

\begin{proof}
From Lemma \ref{lemma: Psi_e_Lipchitz}, $\Psi_e$ is $l_e$-Lipchitz continuous, and hence 
\begin{align*}
\bigg|\sum_{e\in \E}\Psi_e(\c_e^*; Q_{e, \epsilon}) - \Psi_e(\c_e^\mathrm{rnd}; Q_{e, \epsilon})\bigg|
\leq \sum_{e\in \E}l_e|\c_e^* - \c_e^\mathrm{rnd}|
\leq \frac{L}{2}. 
\end{align*}
If $\c_e^{\mathrm{rnd}}$ satisfied the flow conservation constraints, then the above is already the optimality gap. Otherwise, we define a residual vector 
\begin{align*}
\Delta = \I^+{\s} -  \I\c^\mathrm{rnd} \in \mathbb{Z}^{|\V|},
\end{align*}
where a positive value of $\Delta_{v}$ 
indicates a deficit at node $v\in \V$ and a negative value indicates a surplus. 
With balanced supply/demand points, it is trivial to see that $\one^\tran \Delta = 0$, where $\one$ is a conformable vector of 1's. 
Notice that $\I\c^* = \I^+\s$, and thus the 1-norm of the residual vector 
\begin{align*}
\|\Delta\|_1 = \|\I(\c^* - \c^\mathrm{rnd})\|_1 =  
\sum_{v\in \V} \bigg|\sum_{e\in \delta^+(v)}(\c_e^* - \c_e^\mathrm{rnd}) -  
\sum_{e\in \delta^-(v)}(\c_e^* - \c_e^\mathrm{rnd})\bigg|
\leq
\sum_{v\in \V}\frac{\deg(v)}{2} = |\E|.  
\end{align*}
In the heuristic iterations, each unit of added flow travels at most $L$ distance. Thus, 
\begin{align*}
\bigg|\sum_{e\in \E}\Psi_e(\c_e^\mathrm{rnd}; Q_{e, \epsilon}) - \Psi_e(\c_e^\mathrm{fes}; Q_{e, \epsilon})\bigg|
&\leq \frac{\|\Delta\|_1}{2}L \leq \frac{|\E|}{2}L.
\end{align*}
Now, 
from Theorem \ref{theorem: Z_ZCF_error_bounded} and triangle inequality, we have
\begin{align*} 
&
\bigg|\sum_{e\in \E}\Psi_e(\c_e^{\mathrm{fes}}; S_e) - \Psi_e(\c_e^\mathrm{int}; S_e)\bigg| 
= 
\bigg|\sum_{e\in \E}\Psi_e(\c_e^{\mathrm{fes}}; Q_e(\cdot; S_e)) - \Psi_e(\c_e^\mathrm{int}; Q_e(\cdot; S_e))\bigg| \\
\leq &
\bigg|\sum_{e\in \E}\Psi_e(\c_e^\mathrm{fes}; Q_e(\cdot; S_e)) - \Psi_e(\c_e^\mathrm{fes};Q_{e, \epsilon})\bigg|
+
\bigg|\sum_{e\in \E}\Psi_e(\c_e^\mathrm{fes}; Q_{e, \epsilon}) - \Psi_e(\c_e^\mathrm{rnd}; Q_{e, \epsilon})\bigg|\\
&+
\bigg|\sum_{e\in \E}\Psi_e(\c_e^\mathrm{rnd}; Q_{e, \epsilon}) - \Psi_e(\c_e^*; Q_{e, \epsilon})\bigg|
+
\bigg|\sum_{e\in \E}\Psi_e(\c_e^*; Q_{e, \epsilon}) - \Psi_e(\c_e^\mathrm{int}; Q_e(\cdot; S_e))\bigg|
\leq \left(\frac{|\E|}{2} + \frac{1}{2}\right)L + o(1).
\end{align*}
This completes the proof.
\end{proof}

\subsection{Electricity network analogy, and a faster algorithm}
\label{sec: electrict_network}
In this section, 
we show that  
the nonlinear optimality conditions in Theorem \ref{theorem: KKT} 
has a linearized form that mimics a well-known electric network voltage-current problem. 
%
To see this, note that $Q_{e, \epsilon} \in \mathcal{C}^2$ and is bounded, and hence its Taylor expansion at a reference point $x_{e0}\in [0, l_e]$ can be written as
\begin{align*}
c_e^* = Q_{e, \epsilon}\left(\frac{l_e - (\I^\tran\Lam^*)_e}{2}\right) = &  Q_{e, \epsilon}\left(x_{e0}\right) + Q_{e, \epsilon}'\left(x_{e0}\right)\left(\frac{l_e - (\I^\tran\Lam^*)_e}{2} - x_{e0}\right)\\
& + O\left(\left(\frac{l_e-(\I^\tran\Lam^*)_e}{2}-x_{e0}\right)^2\right). 
\end{align*}

For any multiplier vector $\tilde{\Lam}\in \R^{|\V|}$ (e.g., an approximation of $\Lam^*$), define the associated flow $\tilde{\c} = (\tilde{c}_e)_{e \in \E}$ where:
\begin{align*}
\tilde{c}_e = Q_{e, \epsilon}\left(x_{0, e}\right) + Q_{e, \epsilon}'(x_{0, e})\left(\frac{l_e-(\I^\tran\tilde{\Lam})_e}{2} - x_{0, e}\right) , \forall e\in \E.
\end{align*} 
Based on vector $\x_0 := (x_{0, e})_{e\in \E}$, we further define the intercept or ``baseline" vector $\c_0$, the diagonal ``resistance" matrix, 
$\Res$, and the ``weighted Laplacian" matrix $\Lap$ as follows:
\begin{align*}
\c_0(\x_0) = \left(Q_{e, \epsilon}\left(x_{0, e}\right)\right)_{e\in \E}, \quad 
\Res(\x_0) = \operatorname{diag}\!\left( \left(1/Q_{e, \epsilon}'(x_{0, e})\right)_{e\in\E} \right), 
\quad
\Lap(\x_0) =  \I(\Res(\x_0))^{-1}\I^\tran.
\end{align*}
Now, we choose the reference points to be the edges' midpoints (i.e., $\x_0 := \l/2$), and denote 
\begin{align*}
\c_0^* := \c_0(\l/2), \quad \Res^* := \Res(\l/2), \quad
\Lap^* := \Lap(\l/2),
\end{align*}
and then 
\begin{align}
\label{eq: tilde_c}
\tilde{\c} = \c_0^* - \frac{1}{2}(\Res^*)^{-1}\I^\tran\tilde{\Lam}. 
\end{align}
Meanwhile, the
flow conservation constraints $\I\tilde{\c} = \I^+\s$ becomes
\begin{align}
\label{eq: c_tilde}
& \I\c_0^*-\frac{1}{2}\I(\Res^*)^{-1}\I^\tran\tilde{\Lam} = \I^+\s, \quad \Rightarrow\quad 
 \tilde{\Lam} = 2(\Lap^*)^{\dagger} \ (\I\c_0^*-\I^+\s) = 2(\Lap^*)^\dagger (\I\c_0^*-\I^+\s),
\end{align}
where $\dagger$ represents the Moore-Penrose inverse.
Plugging Equation \eqref{eq: c_tilde} back into Equation \eqref{eq: tilde_c}, we have:
\begin{align*}
\tilde{\c} 
=\c_0^* - (\Res^*)^{-1}\I^\tran\left(\Lap^*\right)^\dagger\ (\I\c_0^* - \I^+\s). 
\end{align*}
It is easy to see that $\tilde{\c}$ is the unique solution to a quadratic minimum-energy electricity flow problem: 
\begin{align*}
\min_{\c\in \R^{|\V|}} \quad & \frac{1}{2}(\c-\c_0^*)^\tran \Res^*(\c-\c_0^*),\\
\text{s.t. } \quad &\I\c = \I^+\s,
\end{align*}
and it yields a minimum objective value of 
\begin{align}
\label{eq: linear_min_obj}
\frac{1}{2} (\I\c_0^*-\I^+\s)^\tran(\Lap^*)^\dagger(\I\c_0^*-\I^+\s).
\end{align}
Note that $\c_0^*$ has a natural interpretation in the context of electricity flow problem --- the optimal current solution when the flow conservation constraint is dropped; i.e., in this case, 
the first-order optimality condition gives the optimal solution
\begin{align*}
\c^* = \c_0^*.
\end{align*}
In other words, $\c_e^* = Q_{e, \epsilon}(l_e/2)$ is precisely the uncoupled optimal boundary flow for each edge $e\in \E$. When the conservation constraint $\I\c = \I^+\s$ is imposed, the global optimizer $\c^*$ becomes a coupled perturbation around this local optimum, and the magnitude of this perturbation is governed by the slope $Q_{e, \epsilon}'(l_e/2)$. 

In the bipartite matching context, it quantifies how much ``point mass," captured by $S_e$, lies around its local optimum. 
A large $Q_{e, \epsilon}'(l_e/2)$ implies that the variation of $Q_e$ value near $l_e/2$ is large, and hence the local growth of $F_e$ is small. Since $F_e$ measures how vertex ``imbalance" accumulates along the edge, it acts as a local ``density" of imbalance $S_e$. Thus, a slow growth of $F_e$ represents a relatively even distribution of points over the edge, and small deviation of $\c_e$ incurs a small marginal cost. Hence, we call the term, $Q'_{e, \epsilon}(l_e/2)$, analogous to electricity current conductance, the ``local conductance" of the edge. Its reciprocal is called the ``local resistance" of the edge. 
Besides, we can make other analogies as well. 
Multiplier $\Lam^*$ plays the role of node potential, and its value change across each edge is given by $\I^\tran\Lam^*$. This is similar to the relationship between voltage and current; i.e., the boundary-flow vector $\c^*$ then satisfies the ``Ohm's law:" current = conductance $\times$ voltage drop. Moreover, the conservation constraint becomes the ``Kirchhoff's law," where $-\frac{1}{2}\I(\Res^*)^{-1}\I^\tran\Lam^* = \I^+\s$, stating that the net injection of voltage at nodes equals the divergence of current. 

We next quantify the accuracy of the linear approximation, and shows how the linearization error vanishes as $\|\l\|_\infty$ decreases.
\begin{theorem}
\label{theorem: c*_minus_ctilde_conv_0}
$\|\c^* - \tilde{\c}\|_\Res\to 0$ as $\|\l\|_\infty \to 0$, where $\|\mathbf{\cdot}\|_\Res = \sqrt{(\cdot)^{\mathrm{T}}\Res (\cdot)}$. 
\end{theorem}
 
\begin{proof}
We first claim that for all $e\in \E$, $|(\I^\tran\Lam^*)_e|\leq \|\l\|_\infty$. By Equation \eqref{eq: KKT}, 
\begin{align*}
&0 = 2F_e(\c_e^*; Q_{e, \epsilon}) - l_e + (\I^\tran\Lam^*)_e. 
\end{align*}
Since $F_e(\c_e^*; Q_{e, \epsilon}) \in [0, l_e]$, we obtain $|(\I^\tran\Lam^*)_e|\leq l_e \leq \|\l\|_\infty$. 
By Taylor's expansion, 
\begin{align*}
\c^* = \c_0 - \frac{1}{2}\Res^{-1}\I^\tran\Lam^* + \r,
\end{align*}
where $\r\in \R^{|\E|}$ is the remainder vector of the higher-order term indexed by edges $e\in \E$, and satisfies
\begin{align*}
\r_e \leq \frac{\sup_{x\in [0, l_e]} {Q}''_{e, \epsilon}(x)}{8}(\I^\tran\Lam^*_e)^2
\leq \frac{\max_{e\in \E}\sup_{x\in [0, l_e]} Q''_{e, \epsilon}(x)}{8}\|\l\|_\infty^2, \quad\forall e\in \E. 
\end{align*}
This implies that $\|\r\|_{\Res} \to 0$ as $\|\l\|_\infty \to 0$. Moreover, both $\c^*$ and $\tilde{\c}$ satisfy the flow conservation equation, such that
\begin{align*}
\I\c^* - \I\tilde{\c}
= -\frac{1}{2}\I\Res^{-1}\I^\tran(\Lam^*-\tilde{\Lam}) + \I\r = 0.
\end{align*}
Left multiplying both sides of the second equality by $(\Lam^* - \tilde{\Lam})^\tran$, and rearranging the terms, we have the first equality below:
\begin{align*}
\|\I^\tran(\Lam^*-\tilde{\Lam})\|^2_{\Res^{-1}}
& =
2\left(\sqrt{\Res^{-1}}\I^\tran(\Lam^* - \tilde{\Lam})\right)^\tran\left(\sqrt{\Res} \r\right)\\
& \leq 2 \|\sqrt{\Res^{-1}}\I^\tran(\Lam^* - \tilde{\Lam})\|_2 \ 
\|\sqrt{\Res} \r\|_2
= 2 \|\I^\tran(\Lam^* - \tilde{\Lam})\|_{\Res^{-1}}\|\r\|_{\Res}.
\end{align*}
The second line above follows the 
Cauchy-Schwarz inequality. 
Hence,
\begin{align*}
\|\I^\tran(\Lam^*-\tilde{\Lam})\|_{\Res^{-1}}
\leq 2\|\r\|_{\Res}.
\end{align*}
Finally, since $\c^* - \tilde{\c}
= -\frac{1}{2}\Res^{-1}\I^\tran(\Lam^*-\tilde{\Lam}) + \r$, Cauchy-Schwarz inequality again gives the first inequality below:
\begin{align*}
\|\c^* - \tilde{\c}\|_{\Res}
&\leq
\frac{1}{2}\|\Res^{-1}\I^\tran(\Lam^*-\tilde{\Lam})\|_\Res + \|\r\|_{\Res}
= 
\frac{1}{2}\|\I^\tran(\Lam^*-\tilde{\Lam})\|_{\Res^{-1}} + \|\r\|_{\Res}
\leq 2\|\r\|_{\Res}.
\end{align*}
Thus, as $\|\l\|_\infty\to 0$, $\|\c^*-\tilde{\c}\|_\Res \to 0$. This completes the proof.
\end{proof}

Theorem \ref{theorem: c*_minus_ctilde_conv_0} not only explains how the linearization error decreases with the maximum edge length in the network, but also motivates a simple edge-partition operation that can be used to control the linear approximation error $\|\c^* - \tilde{\c}\|_{\Res}$ to an arbitrary tolerance $\delta$. In so doing, we can simply compute the matrix trace, $\mathrm{tr}(\Res)$, and find a large value $M \geq \frac{\|\l\|_\infty}{2\delta^{1/2}} (\max_{e\in \E}\sup_{x\in [0, l_e]}Q''_{e, \epsilon}(x)\mathrm{tr}(\Res)^{1/2})^{1/2}$.
Then, every edge $e 
\in \E$ can be partitioned into $M$ sub-edges, each with an equal length $l_e/M$, by inserting new nodes to the edge. 
This operation does not change the underlying metric network or the random point realization process, and hence the distribution of the total matching cost remains the same. Yet, by reducing the maximum edge length, Theorem \ref{theorem: c*_minus_ctilde_conv_0} can guarantee that the approximation error bound $2\|\r\|_{\Res}$ does not exceed the tolerance $\epsilon$. 
With this accuracy guarantee, 
we can use $\tilde{\c}$ as a feasible flow solution and evaluate it using the original nonlinear objective function: 
\begin{align}
\label{eq: Z_feas}
\tilde{Z} 
= \sum_{e\in \E} \Psi_e(\tilde{\c}_e; S_e).
\end{align}
Such approximation can be computed in $O(n\log n)$ time with respect to $n$, same as the convex-flow formulation. However, it is typically faster in practice because it avoids iterative Newton updates and computes $\tilde{\c}$ in one step.

\subsection{Extensions}

We now show how the analysis framework can be extended to three variants of the assignment problem. 

The first direct extension is to allow each supply or demand point to carry an arbitrary ``mass" value, such that the problem becomes the so-called transportation problem. 
For instance, a few depots may be regarded as a set of supply point each with a large mass, while multiple customers 
as demand points each with a smaller mass. Under this extension, the model formulation remains essentially unchanged; the only difference is that the magnitudes of the jumps in the imbalance profiles $\{S_e\}$ are now based on the points' masses (rather than unit value).

The second extension is to allow for directed edges in the network. This is natural in transportation networks where one-way traffic is often enforced. 
Recall that in the undirected setting, the orientation of each edge is arbitrarily determined in our model, and the edge flow $\{f_e\}$ is allowed to be positive or negative. For a directed edge $e$, we can simply force the edge orientation to follow 
the prescribed direction, and require
$\{f_e(x)\}$ to be non-negative at every $x$ on the edge. Note that for such an edge $e$,
\begin{align*}
f_e(x) \geq 0 \quad \Rightarrow \quad S_e(x) - c_e \geq 0, \quad  \forall x\in [0, l_e] \quad
\Rightarrow c_e \leq \min_{x\in [0, l_e]} S_e(x).
\end{align*}
As such, directed edges can be incorporated into the convex-flow formulation framework by additionally imposing the above edgewise inequality constraint. 
The objective function and the rest of the formulation remain unchanged. 

The third extension is to allow the two sets of points to have different cardinalities. Let $|U| = n$ and $|V| = m$, and we assume $n \geq m$ without loss of generality. In such a case, only $m$ point pairs can be matched, 
and accordingly the supply-side equality constraints in \eqref{eq: IP} are replaced by inequalities $\sum_{v\in V} y_{uv} \leq 1, \forall u\in U$.
This problem can be easily converted into a balanced matching problem by 
introducing a single dummy demand point with a mass of $(n - m)$, and connecting it to every supply point in $U$ with a zero-distance dummy edge. This augmented network can then be used directly as if the matching problem is now balanced. 

The above transformation notably increases the size of the network because 
$n$ dummy edges must be added, and in turn, 
each supply point must now be treated as a new network node which further increases the number of edges. 
Therefore, as a compromise, the dummy demand may instead be connected to a subset of selected nodes of the underlying network, rather than to all individual supply points, with zero-distance dummy edges. A supply point, if unmatched to a demand point, can be connected to the dummy demand node via detour through the nearest of such selected nodes. 
Under this approximation, the number of dummy connections is reduced to $|\V|$, and the flow-conservation equality in \eqref{eq: CF} is simply replaced by $\I\c \leq \I^+\s$. The approximation error from the detour can be reduced arbitrarily small as each network 
edge is divided into short segments.

\section{Random Bipartite Matching} 
\label{sec: RBMP}
In this section, we consider a random setting in which the points in $U, V$ are i.i.d samples from two probability spaces $(\L, \F, \mu)$ and $(\L, \F, \nu)$,  respectively. Here $\F$ is the $\sigma$-algebra on the network sample-space $\L$, and the measures $\mu$ and $\nu$ specify the sampling distributions of one random point on $\L$; i.e., for any event $A\in \F$ (e.g., a specific portion of 
an edge, or any part of a network), $\mu(A)$ (or $\nu(A)$) is the probability that a random supply (or demand) point lies in $A$. 
Now, $Z^\mathrm{IP}, Z^\mathrm{CF}, Z$ are random variables induced by the random point sets $U$ and $V$, and we are interested in analyzing their expectations. 

\subsection{Asymptotic Convergence}
\label{sec: asy_conv}
Importantly, the realized random problem converge to a limiting problem, which substantially simplifies the analysis. Recall that $n$ is the cardinality of the set $U$ or $V$. For any edge $e$, each random point falls onto this edge $\L_e(x)$ according to a Bernoulli trial, and hence $S_e(x)$ can be interpreted as the difference between the sums of two sets of Bernoulli variables, each of cardinality $n$. The imbalance profile $\{S_e(x): \forall x\in [0, l_e]\}$ has an asymptotic limit when $n \rightarrow \infty$. 
First, when $\mu(\L_e(x)) \not \equiv \nu(\L_e(x))$, by the law of large numbers, as $n\to\infty$, the sample mean of independent random variables converges to the population mean; or equivalently,
\begin{equation}
\begin{aligned}
\frac{S_{e}(x)}{n}
&= 
\frac{1}{n}\left[\sum_{u\in U} \mathbbm{1}(u\in \L_e(x))\right] 
- 
\frac{1}{n}\left[\sum_{v\in V} \mathbbm{1}(v\in \L_e(x))\right]
\xrightarrow{n \to\infty} \mu(\L_e(x)) - \nu(\L_e(x)),
\end{aligned}
\end{equation}
where $\mathbbm{1}(\cdot)$ denotes the indicator function. 
When $\mu(\L_e(x)) \equiv \nu(\L_e(x)) , \forall x$, $S_e(x)$ is a balanced random walk. By Donsker's invariance principle and independence between supply and demand points, if $B$ is a standard Brownian bridge, then $S_e(x)/\sqrt{n}$ converges in distribution as follows: 
\begin{equation}
\label{eq: S_en_to_H_e}
\begin{aligned}
\frac{S_{e}(x)}{\sqrt{n}}
&= \frac{1}{\sqrt{n}}\left[\sum_{u\in U} \mathbbm{1}(u\in \L_e(x))\right] - \frac{1}{\sqrt{n}} \left[\sum_{v\in V} \mathbbm{1}(v\in \L_e(x))\right]\\
& = \frac{1}{\sqrt{n}}\left[\sum_{u\in U} \mathbbm{1}(u\in \L_e(x)) - \mu(\L_e(x))\right] 
- \frac{1}{\sqrt{n}}\left[\sum_{v\in V} \mathbbm{1}(v\in \L_e(x)) - \nu( \L_e(x))\right] \\
&\xrightarrow{n \to\infty} 
\sqrt{2} B(\mu(\L_e(x))).
\end{aligned}    
\end{equation}
As such, the following normalized edge imbalance profile have convergence:
\begin{align*}
\zeta_{e}(x) =
\lim_{n\to\infty}
\begin{cases}
\frac{S_e(x)}{n}  \\
\frac{S_e(x)}{\sqrt{n}}  
\end{cases}
=
\begin{cases}
\mu(\L_e(x)) - \nu(\L_e(x)), & \text{ if } \mu(\L_e(x)) \not \equiv \nu(\L_e(x)), \\
\sqrt{2} B(\mu(\L_e(x))), & \text{ if } \mu(\L_e(x)) \equiv \nu(\L_e(x)). 
\end{cases}
\end{align*}

Meanwhile, recall that (i) the edgewise matching cost $\Psi_e(c_e; S_e)$ is linear with respect to either argument and (ii) Lipschitz continuous with regard to $c_e$, and (iii) the network flow conservation constraints are also linear. Therefore, if we use the limiting imbalance profile $\zeta_{e}(x)$ to replace $S_e(x)$ in \ref{eq: CF} and $\s^\infty := (\zeta_e(l_e))_{e\in \E}$ to replace $\s$ in \ref{conserv}, then the solution to \ref{eq: CF}, denoted $\c^\infty$ with objective value $Z^\infty$, 
will be simply the corresponding limits of normalized $\textbf{c}^\mathrm{int}$ and $Z^\text{CF}$, respectively. In other words, 

\begin{align*}
Z^\infty = 
\lim_{n\to\infty}
\begin{cases}
\frac{Z^\mathrm{CF}}{\sqrt{n}}, & \text{ if } \mu = \nu, \\
\frac{Z^\mathrm{CF}}{n}, & \text{ if } \mu \neq \nu.    
\end{cases}
\end{align*}

\noindent This also implies how $Z^\mathrm{CF}$ asymptotically scales with $n$; i.e., since $Z^\infty$ is clearly positive and finite, we have
\begin{align}
\label{eq: ZCF_scaling}
\Exp[Z^\mathrm{CF}] \asymp
\begin{cases}
\sqrt{n}, & \text{ if } \mu = \nu, \\
n, & \text{ if } \mu \neq \nu.     
\end{cases}
\end{align}

The convergence property also holds for $Q_e$ 
and the associated resistance matrices. First, the following theorem implies that we can use $nQ_e(\cdot; \zeta_e)$ as an asymptotic approximation of $Q_e(\cdot;S_e)$. 

\begin{theorem}
\label{theorem: Q_conv_Qinf}
Assume  every level set of $\zeta_e$ has a zero Lebesgue measure; that is, for every $c\in \R, e\in \E$, 
\begin{align*}
\lambda \{x: \zeta_e(x) = c\} = 0,
\end{align*}
then asymptotic convergence exists as follows:
\begin{align*}
Q_e(x; \zeta_e) = \begin{cases}
\lim_{n\to\infty} \frac{Q_e(x; S_e)}{n}, & \text{ if } \mu(\L_e(x)) \not \equiv \nu(\L_e(x)), \\
\lim_{n\to\infty} \frac{Q_e(x; S_e)}{\sqrt{n}}, & \text{ if } \mu(\L_e(x)) \equiv \nu(\L_e(x)).  
\end{cases}
\end{align*}
\end{theorem}
\begin{proof}
Here we only present the proof when $\mu(\L_e(x)) \not \equiv \nu(\L_e(x))$, since the proof for the other case is similar.  
For any $c\in \R$, 
\begin{align*}
F_e(n c; S_e)
=\int_{x: S_e(x)\le nc}\d x
=\int_{x: S_e(x)/n\le c}\d x.
\end{align*}
When every level set of $\zeta_e$ has a zero Lebesgue measure, the functional $F_e(\cdot; \zeta_e)$ 
is continuous almost surely, i.e., 
\begin{align*}
F_e^-(c; \zeta_e) 
=\int_{x: \zeta_e(x) < c}\d x
=\int_{x: \zeta_e(x) \leq c}\d x
= F_e(c; \zeta_e).
\end{align*}
By the continuous mapping theorem \cite{durrett2019probability}, 
$\lim_{n\to\infty} F_e(nc; S_e) = F_e(c; \zeta_e).
$ 
The normalized inverse function thus has the following convergence property:
\begin{align*}
\frac{1}{n}Q_e(x; S_e) 
&= \frac{1}{n} \inf \{n c: F_e(nc; S_e) \geq x\} 
= \inf \{c: F_e(nc; S_e) \geq x\} \\
&\to \inf \{c: {F}_e(c;\zeta_e) \geq x\} 
= {Q}_e(x; \zeta_e). 
\end{align*}
This completes the proof. 
\end{proof}

Furthermore, if $Q_e(\cdot; \zeta_e)$ is differentiable at $l_e/2$, the resistance matrix $\Res^*$ is well-defined. 
The following lemma states the condition under which this property holds. 
Intuitively, if $\zeta_e$ does not have ``flat" regions around $Q_e(l_e/2; \zeta_e)$, then $F_e(\cdot; \zeta_e)$ will not have jumps at $l_e/2$, and thus $Q_e(\cdot; \zeta_e)$ will be differentiable at $l_e/2$. 
\begin{lemma}
\label{lemma: Q_inf_differentiable}
Suppose the following conditions hold for a twice-differentiable edge profile $\xi_e \in \mathcal{C}^2$,  
\begin{align*}
\xi_e'(x) \neq 0, \quad\forall x: \xi_e(x) = Q_e(l_e/2; \xi_e) = 0,        
\end{align*}
then $Q_e(\cdot; \xi_e)$ is differentiable at $l_e/2$. 
\end{lemma}

\begin{proof}
We will use the coarea formula \citep{coarea} to prove that $F_e(\cdot; \zeta_e)$ is differentiable at $l_e/2$, which states that for an $D$-dimensional open set $\Omega \in \R^D$, a Lipschitz continuous function $f$ and a $L^1$ function $g$, 
\begin{align*}
\int_{\Omega} g(x) |\nabla f(x)| \d x = 
\int_{\R} \left(\int_{f^{-1}(t)} g(x) d\mathcal{H}_{D-1} \right) \d t,
\end{align*}
where $\mathcal{H}_{D-1}$ is the $(D-1)$-dimensional Hausdorff measure. 
When $D = 1$, the above expression reduces 
to 
\begin{align*}
\int_{\Omega} g(x) |f'(x)| \d x = 
\int_{\R} \left(\sum_{x\in f^{-1}(t)} g(x) \right) \d t.
\end{align*}
Let $\xi_e\in \C^2$, $\Omega:= [0, l_e]$, $g(x) := \mathbbm{1}\{a < \xi_e(x) \leq b\} / |\xi_e'(x)|$ for real values $a < b$, and $f(x) := \xi_e(x)$. Then 
\begin{align*}
F_e(b; \xi_e) - F_e(a; \xi_e) &= 
\int_0^{l_e} \mathbbm{1}\{a < \xi_e(x) \leq b\} \d x = 
\int_{\R} \left(\sum_{x: \xi_e(x) = t} \frac{\mathbbm{1}\{a < \xi_e(x) \leq b\}}{|\xi_e'(x)|} \right) \d t \\
&= \int_{a}^b\left(\sum_{x: \xi_e(x) = t} \frac{1}{|\xi_e'(x)|}\right)\d t. 
\end{align*}
Therefore, as $b\rightarrow a$, by the fundamental theorem of calculus:
\begin{align*}
F_e'(Q_e(l_e/2; \xi_e); \xi_e) = \sum_{\{x:\xi_e(x) = Q_e(l_e/2; \xi_e)\}} \frac{1}{|\xi_e'(x)|}> 0.
\end{align*}
Hence, $F_e(\cdot; \xi_e)$ is differentiable at $Q_e(l_e/2; \xi_e)$, and its inverse function $Q_e(\cdot; \xi_e)$ is also differentiable at $l_e/2$ almost surely. This completes the proof. 
\end{proof}

When Lemma \ref{lemma: Q_inf_differentiable} holds, we can define the limiting resistance matrix $\Res^\infty$ as
\begin{align*}
\Res^\infty = \mathrm{diag}\left(\left(\frac{1}{Q_e'(l_e/2; \zeta_e)}\right)_{e\in \E}\right),
\end{align*}
and the normalization immediately gives
\begin{align*}
\Res^\infty_e = 
\begin{cases}
\lim_{n\to\infty}n\Res_e^*, & \text{ if } \mu(\L_e(x)) \not \equiv \nu(\L_e(x)), \\
\lim_{n\to\infty}
\sqrt{n}\Res_e^*, & \text{ if } \mu(\L_e(x)) \equiv \nu(\L_e(x)). 
\end{cases}
\end{align*}
Therefore, the resistance-based formulation in Section \ref{sec: electrict_network} applies directly to asymptotic point distribution profiles. 
The above analysis reveals the fact that resistance scaling differs sharply across the two distributional regimes. When $\mu(\L_e(x)) \not \equiv \nu(\L_e(x))$, the resistance decays at the faster rate of $O(1/n)$. This is not surprising: if supply and demand are distributed differently on the edge, an imbalance appears across edges, and more matched pairs must be routed through network junctions. The effective resistance therefore becomes smaller to allow larger edge flow that accommodates more inter-edge matches. In contrast, when $\mu(\L_e(x)) \equiv \nu(\L_e(x))$ and $n\rightarrow \infty$, most of the matching pairs occur locally, and the edge flow is generated only by negligible stochastic fluctuations. Consequently, the resistance decays more slowly, at rate $O(1/\sqrt{n})$, reflecting a weaker likelihood of forming inter-edge matches -- this property will be discussed further in the next subsection.

\subsection{Stochastic Flow under Identical Distributions}
In this section, we specifically focus on the case $\mu=\nu$ and study the properties of 
$\c^\mathrm{int}$. 
We will show that $\c^\mathrm{int}_e$ is centered, symmetric, and sub-Gaussian for every edge $e$.

\begin{theorem}
\label{theorem: exp_ce*_sq_O(n)}
when $\mu = \nu$, for all $e\in \E$, random variable $\c^\mathrm{int}_e$ has a centered, symmetric, and sub-Gaussian distribution with an $O(n)$ variance proxy.
\end{theorem}

\begin{proof}
By symmetry of inputs in \eqref{eq: CF}, the optimal solution $\c_e^\mathrm{int}$ is clearly centered and symmetric.  
We only need to show that $\c^\mathrm{int}_e$ is a sub-Gaussian random variable. In so doing, we first show that $|\c^\mathrm{int}_e|$ is bounded from above. For any $e\in \E$,
\begin{align*}
\Psi_e(\c^\mathrm{int}_e; S_e)
&=\int_0^{l_e}|S_e(x)-\c^\mathrm{int}_e|\d x 
\geq 
\int_0^{l_e} \left[|\c^\mathrm{int}_e|-|S_e(x)|\right]\d x 
= l_e|\c^\mathrm{int}_e| - \Psi_e(0; S_e). 
\end{align*}
Thus, for each $e\in \E$, 
\begin{align*}
l_e|\c^\mathrm{int}_e| \leq \Psi_e(\c^\mathrm{int}_e; S_e) + \Psi_e(0; S_e)  
\leq \left[\sum_{e\in \E} \Psi_e(\c^\mathrm{int}_e; S_e)\right] + \Psi_e(0; S_e)  
= Z^\mathrm{CF} + \Psi_e(0; S_e),
\end{align*} 
where the second inequality follows nonnegativity of $\Psi_e$. 
Hence, $\c_e^\mathrm{int}$ is bounded from above by $\frac{1}{l_e}[Z^\mathrm{CF} + \Psi_e(0; S_e)]$.

We now study the scaling and concentration properties of $Z^\mathrm{CF}$ and $\Psi_e(0, S_e)$. From Equation \eqref{eq: ZCF_scaling}, we know $\Exp[Z^\mathrm{CF}]$ is $O(\sqrt{n})$. 
From Mcdiarmid's inequality \citep{McDiarmid1989}, it is easy to see $Z^\mathrm{CF} - \Exp[Z^\mathrm{CF}]$ is sub-Gaussian with an $O(n)$ variance proxy. Similarly, 
we can show the properties of $\Psi_e(0; S_e)$. Since
$\Psi_e(0;S_e) = \int_{0}^{l_e} |S_e(x)|$, we first analyze the properties of $S_e(x)$, which is a sum of $2n$ independent random variables taking value from $\{-1, 0, 1\}$. Each of these variables 
has a variance no larger than $1/4$, and hence: 
\begin{align*}
\mathrm{Var}(S_e(x)) = \mathrm{Var}\left\{\left[\sum_{u\in U} \mathbbm{1}(u \in \L_e(x))\right] -\left[ \sum_{v\in V} \mathbbm{1}(v\in \L_e(x)) \right]\right\}
= 2n \mathrm{Var}(\mathbbm{1}(u\in \L_e(x))) \leq \frac{n}{2}.
\end{align*}
Therefore by Cauchy-Schwarz inequality \citep{durrett2019probability}, $\Exp[|S_e(x)|] \leq \sqrt{n/2} = O(\sqrt{n})$, which implies that $\Exp[\Psi_e(0;S_e)] = O(\sqrt{n})$. 
Finally, adding or removing one supply or demand point changes $S_e(x)$ by at most 1 anywhere over $[0, l_e]$, and therefore,  by McDiarmid's inequality \citep{McDiarmid1989}, $\Psi_e(0) - \Exp[\Psi_e(0)]$ is sub-Gaussian with an $O(n)$ variance proxy. 

Therefore, if we denote $Y = Z^\mathrm{CF} + \Psi_e(0; S_e)$, we can conclude that $\Exp[Y] = O(\sqrt{n})$ and $Y - \Exp[Y]$ is sub-Gaussian with an $O(n)$ variance proxy; i.e., 
there exists a positive constant $C_1$ such that $\Exp[Y] \leq C_1\sqrt{n}$, and for all $t \geq 0$, there exists a positive constant $C_2$ such that: 
\begin{align}
\label{eq: P_c_e*_geq_t_leq_P_Y_geq_let}
\P(|\c_e^\mathrm{int}| \geq t) \leq \P(Y \geq l_e t)
= \P(Y - \Exp[Y] \geq l_e t - \Exp[Y])
\leq \exp\left(-\frac{(l_e t - \Exp[Y])^2}{C_2n}\right).
\end{align}

It remains to be shown that for every $t\geq 0$, there exists a constant $C_3 > 0$ such that 
\begin{align}
\label{eq: c_e*_subG_ineq}
\P(|\c_e^\mathrm{int}| \geq t) \leq 2\exp\left(-\frac{t^2}{C_3n}\right).
\end{align}
We consider two cases. 
For $t\geq \frac{2C_1\sqrt{n}}{l_e}$, we have 
$l_e t - \Exp[Y] \geq \frac{1}{2}l_e t$ so that 
\begin{align*}
\P(|\c_e^\mathrm{int}| \geq t) \leq 
\P(Y-\Exp[Y]\geq l_et-\Exp[Y]) \leq \exp\left(-\frac{l_e^2t^2}{4C_2n}\right) 
\quad \Leftrightarrow \quad 
C_3 \geq \frac{4C_2}{l_e^2}.
\end{align*} 
For $0 < t < \frac{2C_1\sqrt{n}}{l_e}$, we choose $C_3$ large enough to satisfy 
\begin{align*}
2\exp\left(-\frac{t^2}{C_3n}\right)\geq 1
\quad \Leftrightarrow \quad 
C_3 \geq \frac{4C_1^2}{l_e^2\log 2}.
\end{align*}
Since a probability can at most be 1, Equation \eqref{eq: c_e*_subG_ineq} holds trivially 
as long as we choose $C_3 = \max\{\frac{4C_2}{l_e^2}, \frac{4C_1^2}{l_e^2\log 2}\}$. 
This completes the proof. 
\end{proof}

A direct corollary of Theorem \ref{theorem: exp_ce*_sq_O(n)} --- which is formal proof for the scaling properties in Section \ref{sec: asy_conv} 
--- is that the fraction of inter-edge matches diminishes as $n\to\infty$. 
    
\begin{corollary}
\label{col: n_inter_to_0}
when $\mu = \nu$, the fraction of inter-edge matches converges to zero as $n\to\infty$. 
\end{corollary}
\begin{proof}
Since $\Exp[|\c_e^\mathrm{int}|] \leq \sqrt{\Exp[|\c_e^\mathrm{int}|^2]} =  O(\sqrt{n})$, the corresponding flow $f_e$ at position $(e, 0)$ satisfies: 
\begin{align*}
\Exp[|f_e(0)|] = \Exp[|\c_e^\mathrm{int}|] = O(\sqrt{n}).  
\end{align*}
By symmetry, the flow at position $(e, l_e)$ has the same scaling $\Exp[|f_e(l_e)|] = O(\sqrt{n})$. Since these two flow values dictate the number of points on edge $e$ that are matched ``outside of" $e$, then, 
\begin{align*}
\frac{1}{n}\Exp[|f_e(0)| + |f_e(l_e)|] = O\left(\frac{1}{\sqrt{n}}\right) \xrightarrow{n\to\infty} 0.
\end{align*}
Summing the above across all edges, and the Corollary holds.
\end{proof}

As a final remark, the above properties provide theoretical support for the approximate network matching distance formulas in \cite{zhai2026average}, which assume that the point matches are 
primarily local (within the same edge), and 
only a negligible fraction of matches are between points from different edges. The observations in the numerical experiments in that reference also indirectly corroborate these properties. 


\subsection{A Fast Approximate Solution Algorithm 
under Nonidentical Distributions}
As established in Section \ref{sec: asy_conv}, when $\mu(\L_e(x)) \not \equiv \nu(\L_e(x))$, the normalized imbalance profile converges asymptotically to a limiting profile and leads to a limiting flow formulation,  implying that the effect of sampling randomness vanishes asymptotically. This observation motivates the use of the limiting-flow formulation as a computationally inexpensive approximation for finite realized instances. However, the condition $\mu(\L_e(x)) \not \equiv \nu(\L_e(x))$ need not hold on every edge. We therefore first show that the stochastic contribution by those unqualified edges are asymptotically dominated by the deterministic flow on the remaining edges, and suitable regulation terms can be used to approximate the solution. 

Recall that
the resistance of each edge $\Res_e^*$ exhibits different asymptotic scalings according to the local imbalance $\mu(\L_e(x))$ and $\nu(\L_e(x))$: $\Res_e^*$ is in order of $O(1/n)$ when $\mu(\L_e(x))\not\equiv \nu(\L_e(x))$, and in order of $O(1/\sqrt{n})$ otherwise. 
Thus, edges satisfying the condition $\mu(\L_e(x))\not\equiv \nu(\L_e(x))$ have a smaller resistance and can accommodate larger network-level flows.
Such behavior gives rise to two asymptotic regimes: when $\mu = \nu$ for the entire network, the optimal flow and matching cost are governed by stochastic fluctuations of order $O(\sqrt{n})$; when $\mu \neq \nu$ in the network, those edges with $\mu(\L_e(x)) \not\equiv \nu (\L_e(x))$ contribute most deterministic imbalance and asymptotically dominate the $O(\sqrt{n})$ stochastic fluctuation, resulting in network-level optimal flow and total matching cost to be on the order of $n$.
When $\mu \neq \nu$, some edges may be locally balanced, i.e., satisfying $\mu(\L_e(x))\equiv \nu (\L_e(x))$; for each of them, we can introduce a sufficiently small linear regulation term with constant $\delta > 0$ that impose non-identical point distributions: 
\begin{align*}
\mu(\L_e(x)) - \nu(\L_e(x)) = \delta x.
\end{align*}
This regularization yields a non-degenerate limiting flow problem, from which both $\c^\infty$ and $\Res^\infty$ are well-defined. 

The resulting limiting resistance can then be used to construct a fast approximate solution algorithm next. 
%
%
As established in Section \ref{sec: asy_conv}, the normalized optimal flow has the asymptotic convergence $\c^\mathrm{int}/n \to \c^\infty$. Equivalently $\c^\mathrm{int} = n\c^\infty + o(n)$, indicating that 
$n\c^\infty$ captures the dominant $O(n)$ component of the optimal flow. 
However, $n\c^{\infty}$ generally does not satisfy the constraint $\I (n\c^{\infty}) = \I^+\s$.
Therefore, we use $n\c^\infty$ as the baseline and approximate the remaining lower-order perturbation using a single feasibility projection. 
We define the one-shot normalized flow $\c^\mathrm{1s}$ as the $\Res^\infty$-weighted projection of $\c^{\infty}$ onto the realized feasible space:
\begin{equation}
\begin{aligned}
\c^\mathrm{1s} = \ \argmin{\c\in \R^{|\E|}} \quad &\frac{1}{2} (\c - \c^{\infty})^\tran \Res^\infty (\c - \c^{\infty}),
\text{    s.t.} \, & \I\c = \I^+ \frac{\s}{n}.
\end{aligned}    
\end{equation}
This projection chooses the feasible flow that requires the smallest resistance-weighted correction from the limiting optimizer. Its closed-form solution is 
\begin{align*}
\c^{\mathrm{1s}} = \c^{\infty} + (\Res^\infty)^{-1} \I^\tran [\I(\Res^\infty)^{-1} \I^\tran]^\dagger \left(\I^+\frac{\s}{n} - \I\c^{\infty}\right) = \c^{\infty} + (\Res^\infty)^{-1} \I^\tran [\I(\Res^\infty)^{-1} \I^\tran]^\dagger \I^+ \left(\frac{\s}{n} - \s^\infty\right).
\end{align*}
The second equality above holds because $\I\c^{\infty} = \I^+\s^\infty$. 
Thus, $n\c^{\mathrm{1s}}$ satisfies the original flow conservation constraint and provides a feasible estimate of the optimal objective value: 
\begin{align*}
Z^\mathrm{1s} = \sum_{e\in \E} \Psi_e(n\c^\mathrm{1s}; S_e). 
\end{align*}
It is observed that all quantities in the projection formula are determined by the underlying distributional profiles, except for the realized imbalance vector $\s$. The discrepancy between $\s/n$ and $\s^\infty$ induces a correction to the limiting flow $\c^\infty$. This leads to a one-shot asymptotic approximation algorithm: the distribution-dependent quantities can be precomputed offline, and for each realized instance, the feasible flow is obtained by a single resistance-weighted projection. 
Such an offline-online decomposition scheme makes the one-shot approximation particularly useful when the matching problem is repeatedly solved on the same network under stable distributions. Once the limiting flow and the resistance matrix are precomputed, each new realization requires only the updated imbalance vector and a single projection, substantially reducing the online computational burden.

\section{Numerical Experiments}
\label{sec: numerical}

This section presents numerical experiments to evaluate the proposed formulations and examine the theoretical properties of random bipartite matching on general networks. We first describe the experimental settings, including the tested networks, point distributions, and implementation details in Section \ref{sec: experiment_setting}. We then compare the accuracy and computational performance of deterministic estimators $Z^\mathrm{IP}, Z, \tilde{Z},$ and $Z^\mathrm{1s}$ in Section \ref{sec: numerical_deter}. 
Last, in Section \ref{sec: numerical_random}, we validate the asymptotic scaling of the total matching cost, and investigate the fluctuation of the optimal boundary flow to demonstrate asymptotically vanishing fractions of inter-edge matching. 

\subsection{Experimental settings}
\label{sec: experiment_setting}
We validate the accuracy of the proposed network formulations using a series of Monte-Carlo simulations. The tested networks are selected to represent increasing levels of topological complexity and size: a circular network, grid networks with different edge numbers (e.g., $4\times 4$ and $31\times 31$), the Sioux Falls network, and the Chicago sketch network \citep{transportationNetworks}. Figure \ref{fig: vis_graphs} illustrates the corresponding network structures and the number of edges and nodes. 

\begin{figure}[htbp]
    \centering
    \includegraphics[width=1\linewidth]{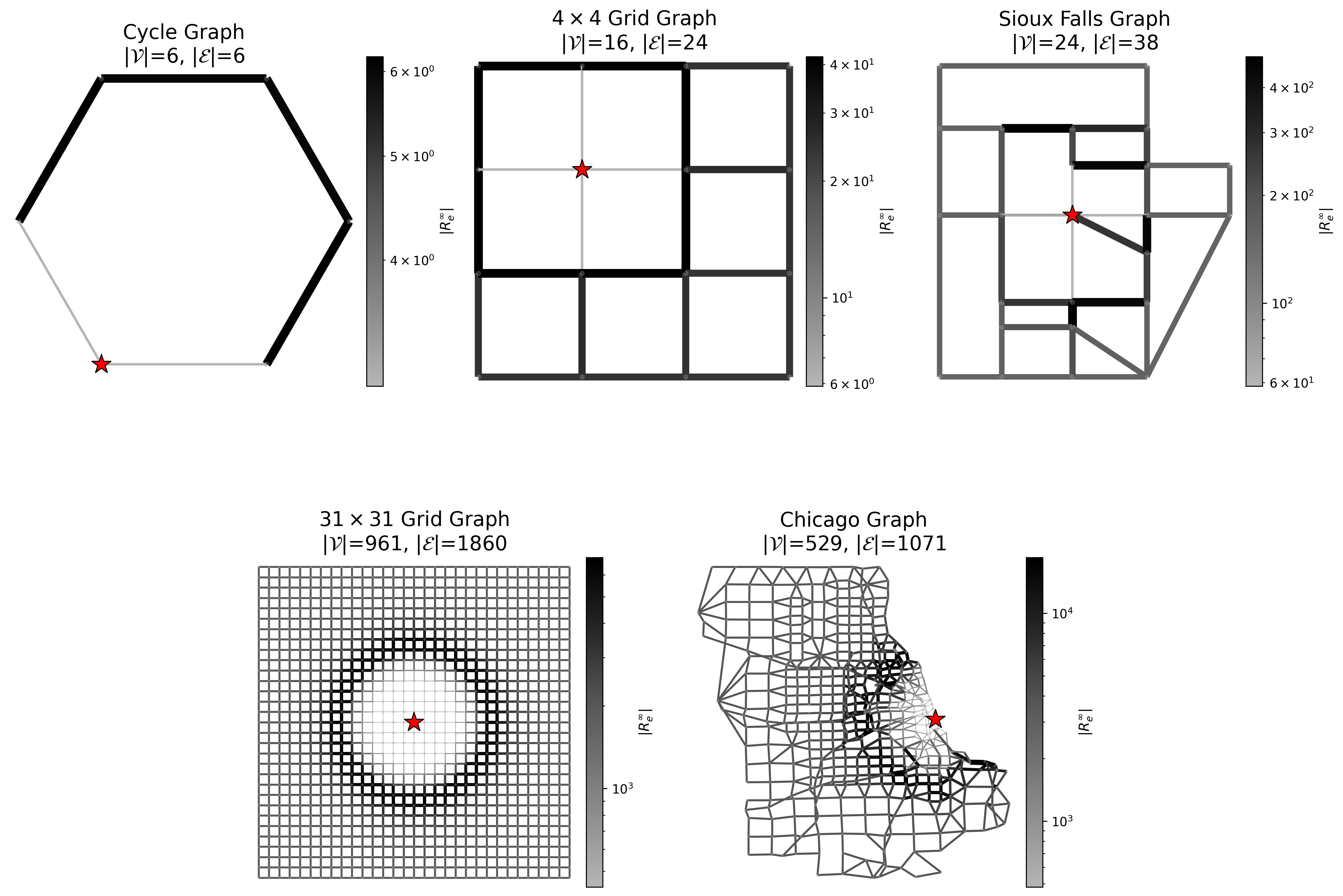}
    \caption{Test graphs with geometric centers and limiting resistance.}
    \label{fig: vis_graphs}
\end{figure}

For the small-scale networks (e.g., circular, $4\times 4$ grid, and Sioux Falls), we vary $n \in $\{100, 500, 1000, 1500, 2000\}. For larger-scale networks,
we vary $n \in $ \{500, 1000, 2000, 5000,10000, 20000\}. For large instances, e.g., when $n\geq 5000$, the standard algorithm becomes computationally expensive. Therefore, we omit $Z^\mathrm{IP}$ for these cases and use $Z$ as the benchmark, based on the consistency with $Z^\mathrm{IP}$ in smaller instances.  For each instance, $n$ supply points and $n$ are demand points are randomly generated on the network edges according to predetermined edge-wise density distributions.
If we need to generate instances under $\mu = \nu$, we simply let supply and demand points to both be uniformly distributed everywhere along all network edges with equal probabilities. 
If we need to generate instances under $\mu \neq \nu$, we still let supply points to be uniformly distributed everywhere in the network. However, we assign demand points heterogeneously to different edges 
based on the edges' relative distances to an arbitrarily chosen ``center'' of the network (the red star markers in 
Figure \ref{fig: vis_graphs}), as follows: 
\begin{align*}
\nu(A) = \sum_{e\in \E} \frac{\exp(-\beta d_e)}{\sum_{e'\in \E} \exp(-\beta d_{e'})}\frac{\lambda(A \cap \L_e)}{l_e}, \quad 
\forall A \in \F, 
\end{align*}
where $\beta > 0$ is a parameter that captures the relative level of heterogeneity. 
In other words, a demand point first ``chooses'' an edge $e$ with probability $\exp(-\beta d_e)/\sum_{e'\in \E} \exp(-\beta d_{e'})$, and is then positioned uniformly along that edge. 
The limiting profile and the limiting resistance are easily computed as follows:
\begin{align*}
&\zeta_e(x) = \left|\frac{1}{L} - \frac{\exp(-\beta d_e)}{l_e\sum_{e'\in \E} \exp(-\beta d_{e'})}\right|x, \quad \forall e\in \E, \forall x\in [0, l_e],\\
&\Res^\infty = \mathrm{diag}\left(\left|\frac{1}{L} - \frac{\exp(-\beta d_e)}{l_{e}\sum_{e'\in \E} \exp(-\beta d_{e'})}\right|^{-1}_{e\in \E}\right).
\end{align*}
Figure \ref{fig: vis_graphs} visualizes the resistance of each edge when $\beta = 10$, with darker and thicker edges indicating larger resistance values. 
Edges near and very far away from the center generally exhibit smaller resistance because their larger supply-demand imbalance requires greater inter-edge flow 
across edges. Somewhere the middle range of distance, the demand density approximately equals the uniform supply density, and the resistance tends to be larger. 

For each combination of $n$, network, and distribution, 100 independent realizations are generated and the optimal total matching distance $Z^\mathrm{IP}$ is computed by the state-of-art JV algorithm \citep{JV_algorithm} as a benchmark. The sample mean across 100 realizations is recorded as the average optimal matching distance. The computation time is also recorded.  
Meanwhile, to solve the convex flow-based formulations, we use the following smooth approximations: 
\begin{align*}
&\Psi_e(c; Q_{e, \epsilon}) = \int_{0}^{l_e} \sqrt{(S_e(x) - c)^2 + \epsilon^2} \d x,\\
&F_e(c; Q_{e, \epsilon}) = \frac{\Psi_{e}'(c; Q_{e,\epsilon})+l_e}{2} = \frac{l_e}{2} + \frac{1}{2}\int_{0}^{l_e} \frac{c-S_e(x)}{\sqrt{(S_e(x) - c)^2 + \epsilon^2}} \d x, \, \text{and }
Q_{e, \epsilon} = F_e^{-1}(\cdot; Q_{e, \epsilon}),
\end{align*}
where $Q_{e,\epsilon}\to Q_{e}$ as $\epsilon \to 0$. 
The smaller the value of $\epsilon$, the closer the approximation of edge profile. 


\subsection{Verification of deterministic estimators}
\label{sec: numerical_deter}
Unless otherwise specified, we set $\beta=10$, use $\epsilon=0.1$ for computing $Z$ and $Z^{\mathrm{1s}}$, and use $\epsilon=1$ for computing $\tilde{Z}$. The smaller value used for $Z$ and $Z^{\mathrm{1s}}$ provides a closer approximation to the original non-smooth edge potential, whereas the larger value used for $\tilde{Z}$ produces a smoother rearranged profile and a more stable local linearization. The sensitivity of these estimators to these choices is examined later in Section \ref{sec: epsilon_sensitivity}.

\subsubsection{Small-scale networks}
We start with the small-scale networks. Figures \ref{fig: three_network_cost} and \ref{fig: three_network_error} show the objective values and relative errors for the small-scale networks. 
$Z^\mathrm{IP}$, $Z$, $\tilde{Z}$, and $Z^\mathrm{1s}$ are represented by black solid curves, red dashed curves with circle markers, green dotted curves with square markers, and blue dash-dot curves with triangle markers, respectively. First, it is observed that across all cases, the estimation by $Z$ closely matches the benchmark $Z^\mathrm{IP}$, with average relative errors smaller than 1\%; moreover, the error between them generally decreases as $n$ increases. 
Second, the first-order approximation $\tilde{Z}$ also provides very accurate estimates in most cases, with relative errors generally below 5 \%, except for the non-identical distributions case on the Sioux Falls network. For the non-identical cases, the relative error of $\tilde{Z}$ appears to increase as the network topology becomes more complex. This pattern arises because the first-order approximation relies on local linearization of $Q_e(\cdot; S_e)$, which is less accurate when the flow varies more across a complex network. Finally, we examine the one-shot fast approximation algorithm. The results show that $Z^\mathrm{1s}$ becomes accurate with average errors smaller than 5\% when $n$ is sufficiently large, e.g., $n\geq 500$. Its relative error consistently decreases as $n$ becomes larger, which is consistent with our expectation.

\begin{figure}[htbp]
    \centering
    \includegraphics[width=0.95\linewidth]{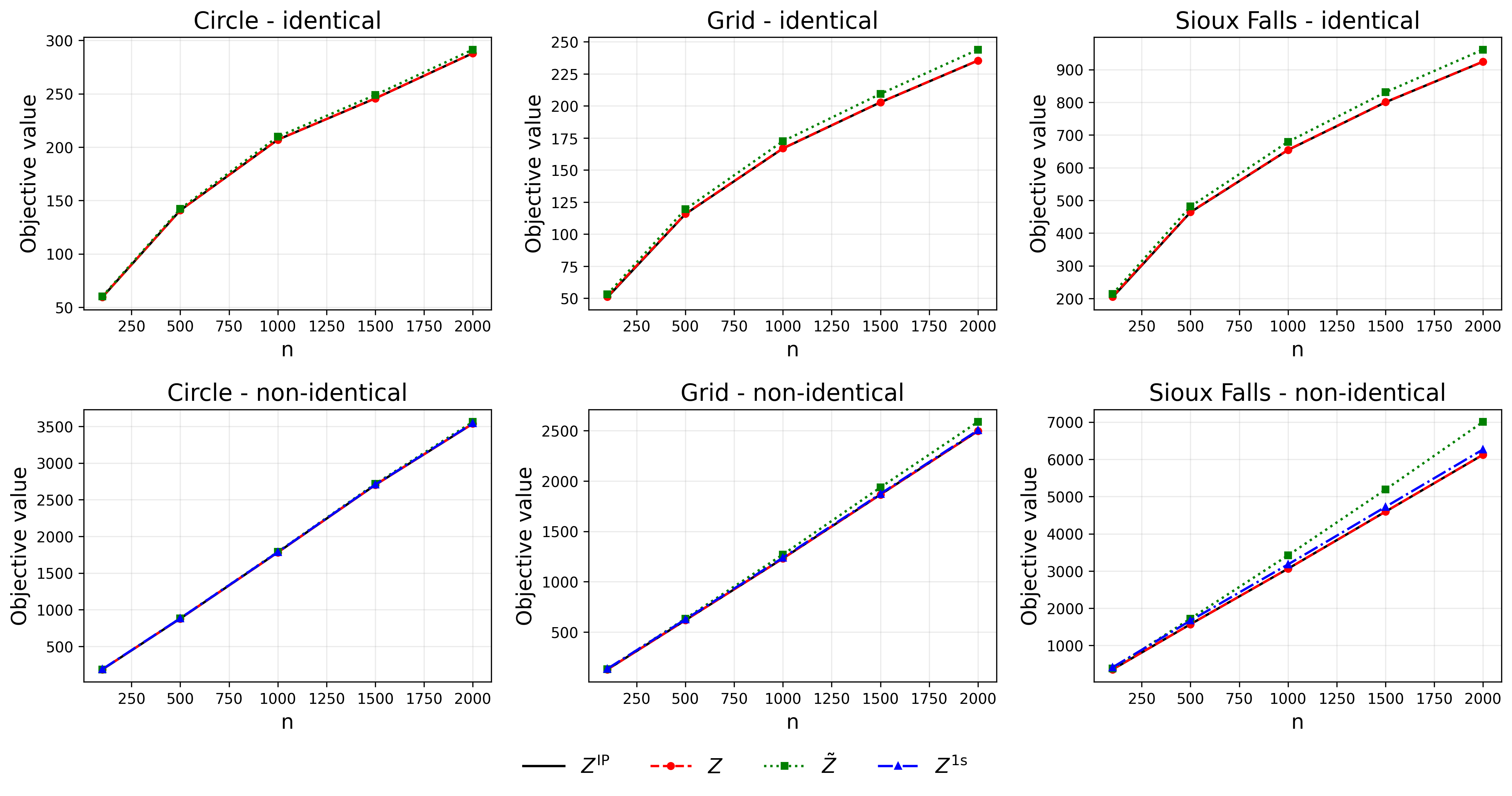}
    \caption{Comparison of objective values for small-scale networks. }
    \label{fig: three_network_cost}
\end{figure}

\begin{figure}[htbp]
    \centering
    \includegraphics[width=0.95\linewidth]{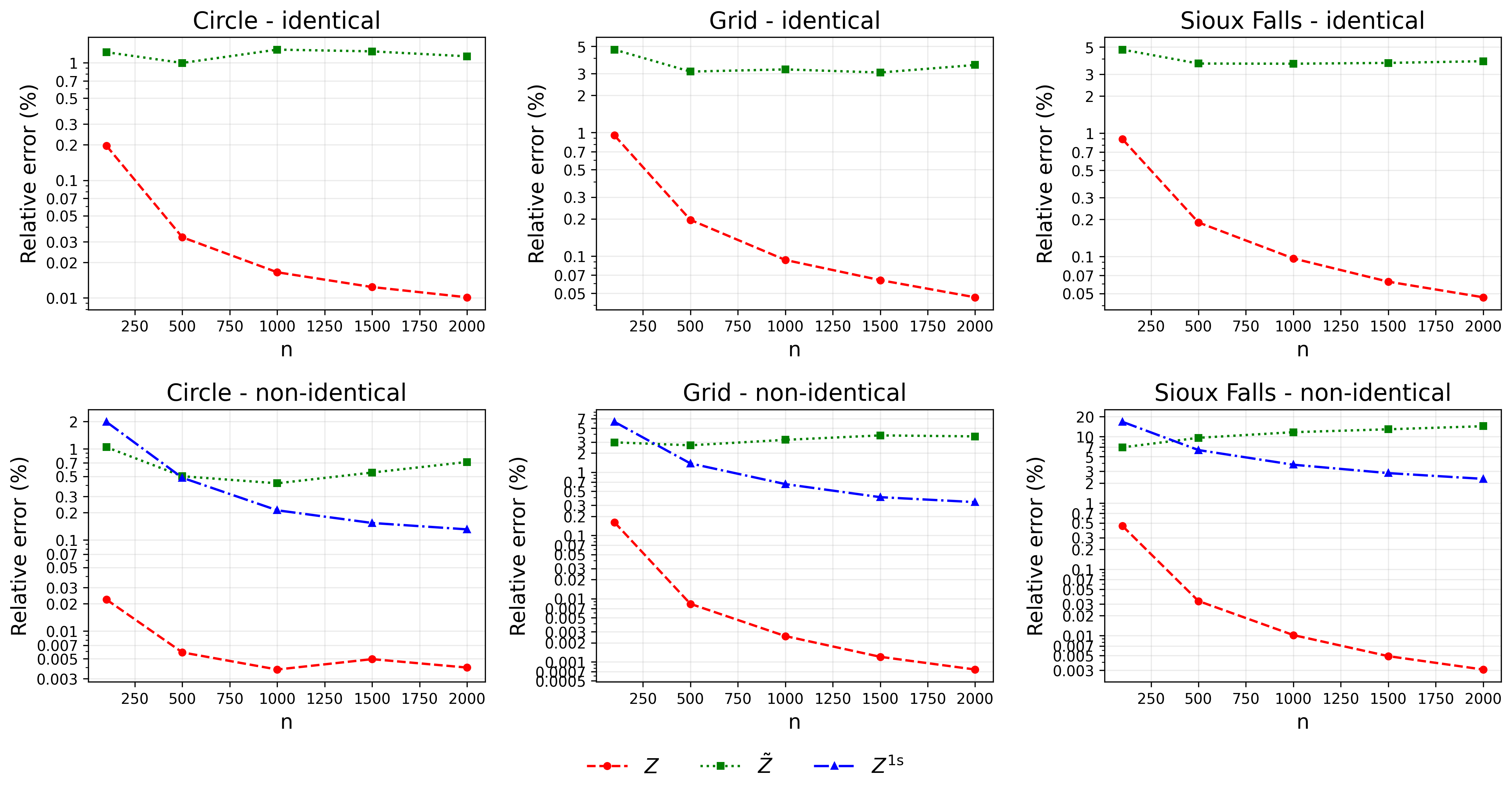}
    \caption{Comparison of relative errors for small-scale networks. }
    \label{fig: three_network_error}
\end{figure}

Next, we compare the computational performance of the four methods. 
Figures~\ref{fig: three_network_time} shows the average running times of $Z^\mathrm{IP}$, $Z$, $\widetilde Z$, and $Z^\mathrm{1s}$ in all realizations for the small networks. 
The average running times of $Z$, $\widetilde Z$, and $Z^\mathrm{1s}$ remain below $0.03$, $0.015$, and $0.005$ seconds, respectively, over the tested range of $n$. 
Their running times increase only slightly with $n$, because the dimensions of the corresponding flow and Laplacian systems depend primarily on the size of the network rather than on $n$. 
In contrast, the running times of $Z^\mathrm{IP}$ increases rapidly with $n$ because it requires creating the pairwise shortest-path cost matrix and solving the resulting assignment problem. 
Hence, although solving $Z^\mathrm{IP}$ directly can be computationally efficient for very small instances, the flow-based estimators become substantially faster for sufficiently large $n$, e.g., $n\geq 500$. Among the proposed methods, $\widetilde Z$ is even faster than $Z$ because it replaces the iterative Newton procedure with a single Laplacian solve, but $Z^\mathrm{1s}$ achieves the fastest online computation time by using precomputed limiting solutions.

\begin{figure}[htbp]
\centering    \includegraphics[width=0.95\linewidth]{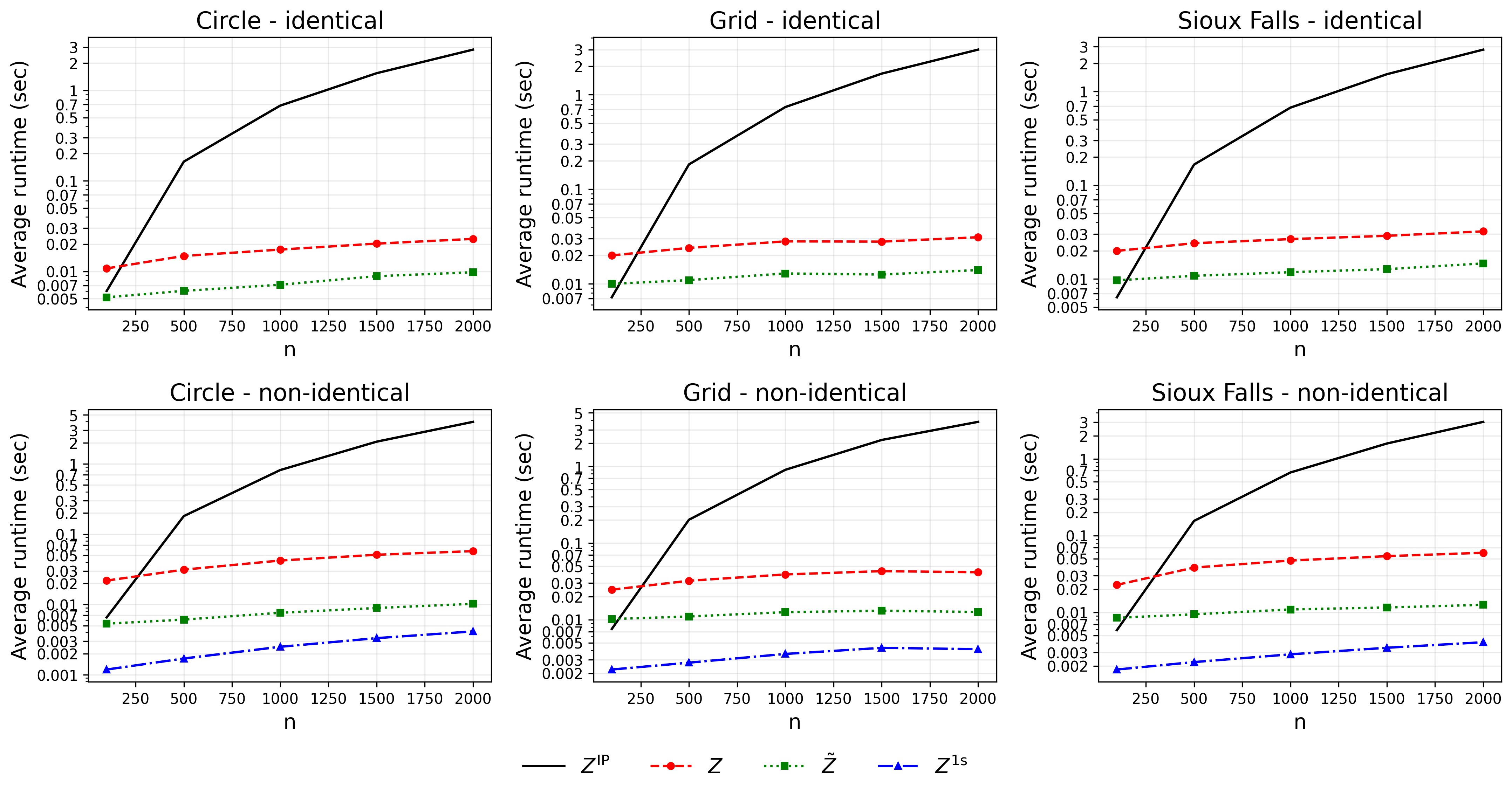}
    \caption{Comparison of running time for small-scale networks. }
    \label{fig: three_network_time}
\end{figure}

\subsubsection{Large-scale networks}
We next validate the accuracy of estimators on larger-scale networks.
Figures \ref{fig: large_network_cost} and \ref{fig: large_network_error} show the results. All estimators are observed with trends similar to those observed on smaller networks: $Z$ provides the most accurate prediction; $\tilde{Z}$ is most accurate when the network structure is simple and the distributions are identical; $Z^\mathrm{1s}$ is effective mainly in the non-identical point distribution regime when $n$ is large. 
\begin{figure}[htbp]
    \centering
    \includegraphics[width=0.95\linewidth]{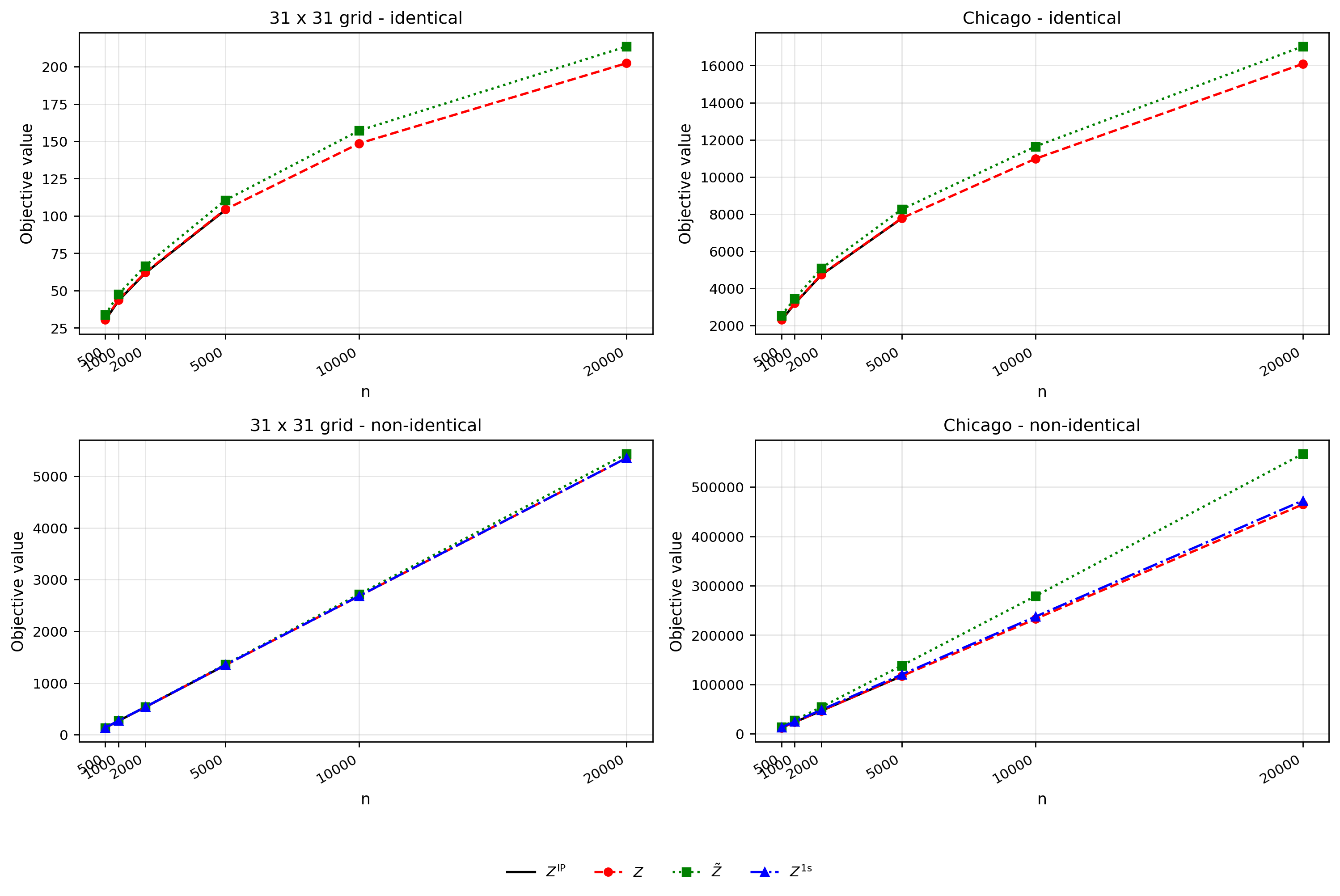}
    \caption{Comparison of objective values for for large-scale networks. }
    \label{fig: large_network_cost}
\end{figure}

\begin{figure}[htbp]
    \centering
    \includegraphics[width=0.95\linewidth]{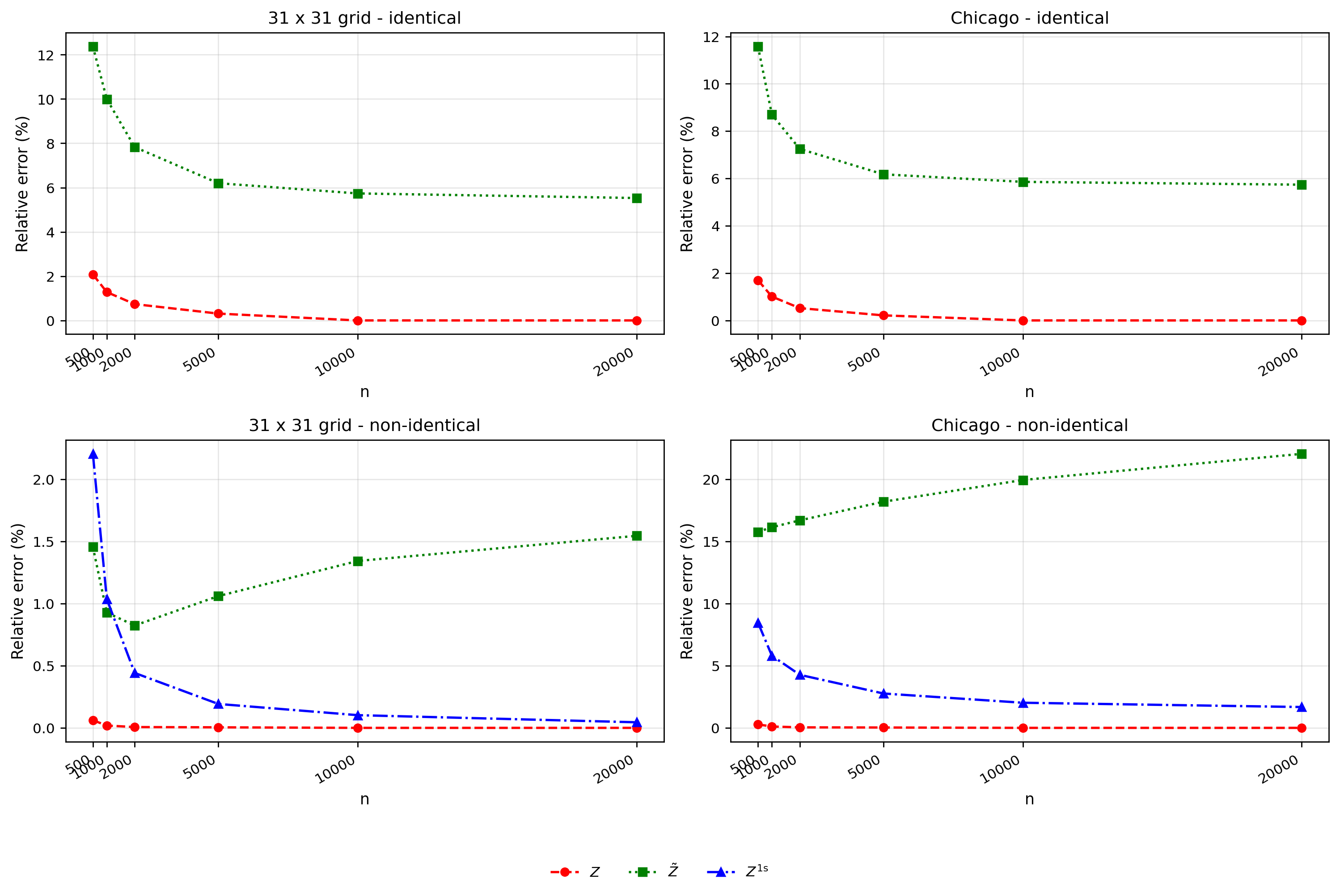}
    \caption{Comparison of relative errors for for large-scale networks. }
    \label{fig: large_network_error}
\end{figure}

For the computational performance of the four methods, the average running times of $Z, \tilde{Z},$ and $Z^\mathrm{1s}$ remain below approximately 15, 0.5, and 0.1 seconds, respectively, over the tested range of $n$. Although the larger network size increases the dimensions of the flow and Laplacian systems, the running time of the proposed estimators still grow only mildly with $n$. In contrast, the computational cost of $Z^\mathrm{IP}$ increases rapidly, and becomes expensive when $n$ is large, e.g., $n\geq 5000$. These results demonstrate that $\tilde{Z}$ and $Z^\mathrm{1s}$ provide substantial computational advantages on larger networks, with $Z^\mathrm{1s}$ reaching the shortest online running time under non-identical point distributions.

\begin{figure}[htbp]
    \centering
    \includegraphics[width=0.95\linewidth]{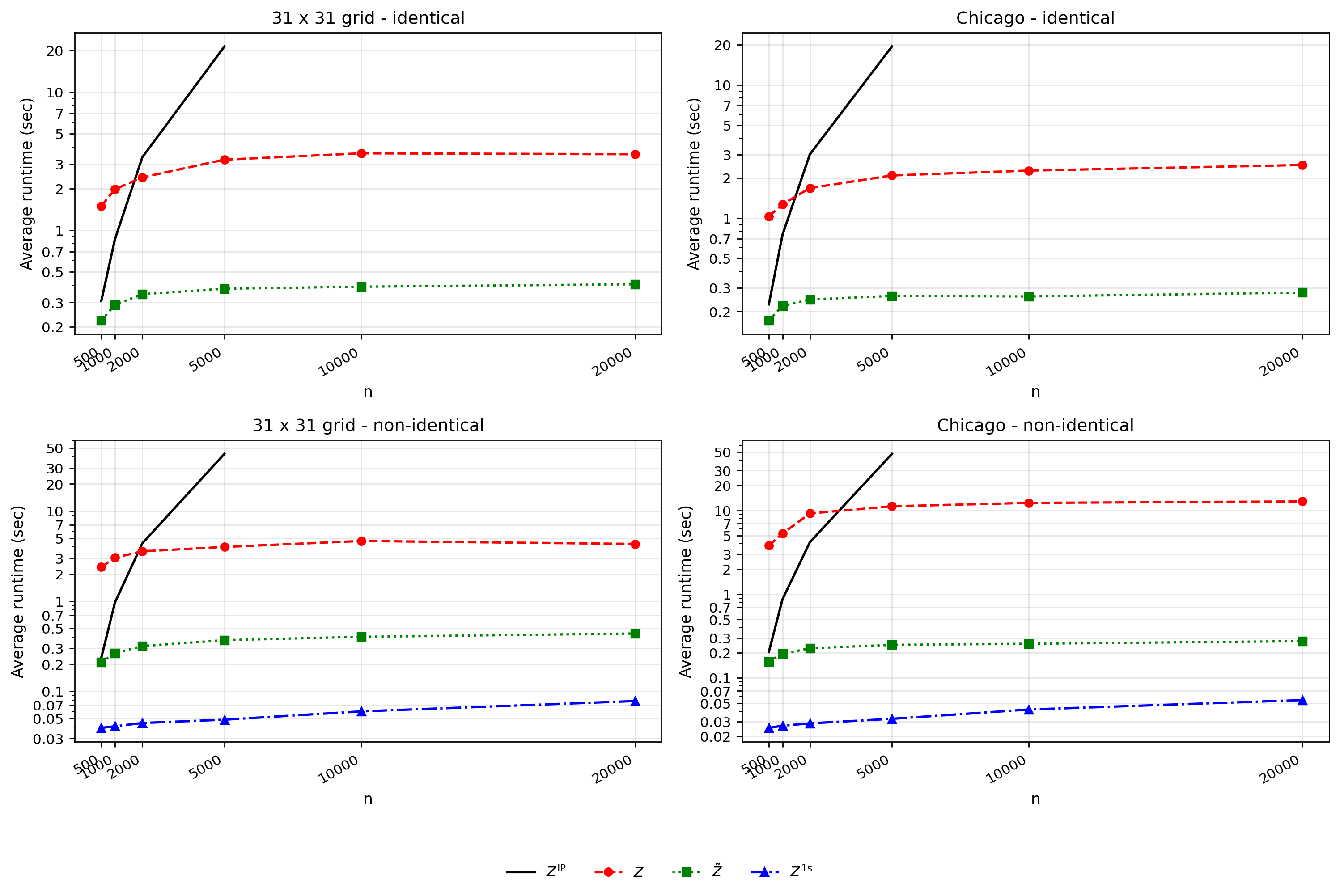}
    \caption{Comparison of running time for  large-scale networks. }
    \label{fig: large_network_time}
\end{figure}

\subsubsection{Sensitivity to parameter $\epsilon$}
\label{sec: epsilon_sensitivity}
We also conduct sensitivity analyses to examine how the choice of parameter $\epsilon$ affects the accuracy and computational efficiency of the estimators. We use the Sioux Falls network and fix $n = 500$. We vary $\epsilon$ over \{0.01, 0.05, 0.1, 0.25, 0.5, 0.75, 1, 1.5, 2, 5, 10\}. 

Figure \ref{fig: error_eps} shows the relative error of $Z$ and $\tilde{Z}$, as compared with $Z^\mathrm{IP}$. The results reveal a clear approximation tradeoff. The relative error of $Z$ remains close to zero whenever $\epsilon$ is small, but increases as $\epsilon$ grows, as expected. 
In contrast, the error of $\tilde{Z}$ decreases substantially with $\epsilon$, since the modified edge objective becomes increasingly quadratic and is therefore better represented by first-order approximations. 

Figure \ref{fig: time_eps} shows the running time of the three methods for different $\epsilon$. The running time of $Z^\mathrm{IP}$ is independent of $\epsilon$ and therefore remains constant. In contrast, the running time of computing $Z$ decreases substantially under a large $\epsilon$, e.g., $\epsilon \geq 0.25$. This occurs because a larger $\epsilon$ produces a smoother and better-conditioned edge objective, allowing the Newton iteration to converge more rapidly. 
This indicates the computational tradeoff: a smaller $\epsilon$ value provides a more accurate approximation of the original objective but requires a longer computation time, and vice versa. 
On the other hand, the running time of $\tilde{Z}$ is consistently the lowest and remains nearly unchanged across different $\epsilon$ values.  
\begin{figure}
    \centering
    \includegraphics[width=0.9\linewidth]{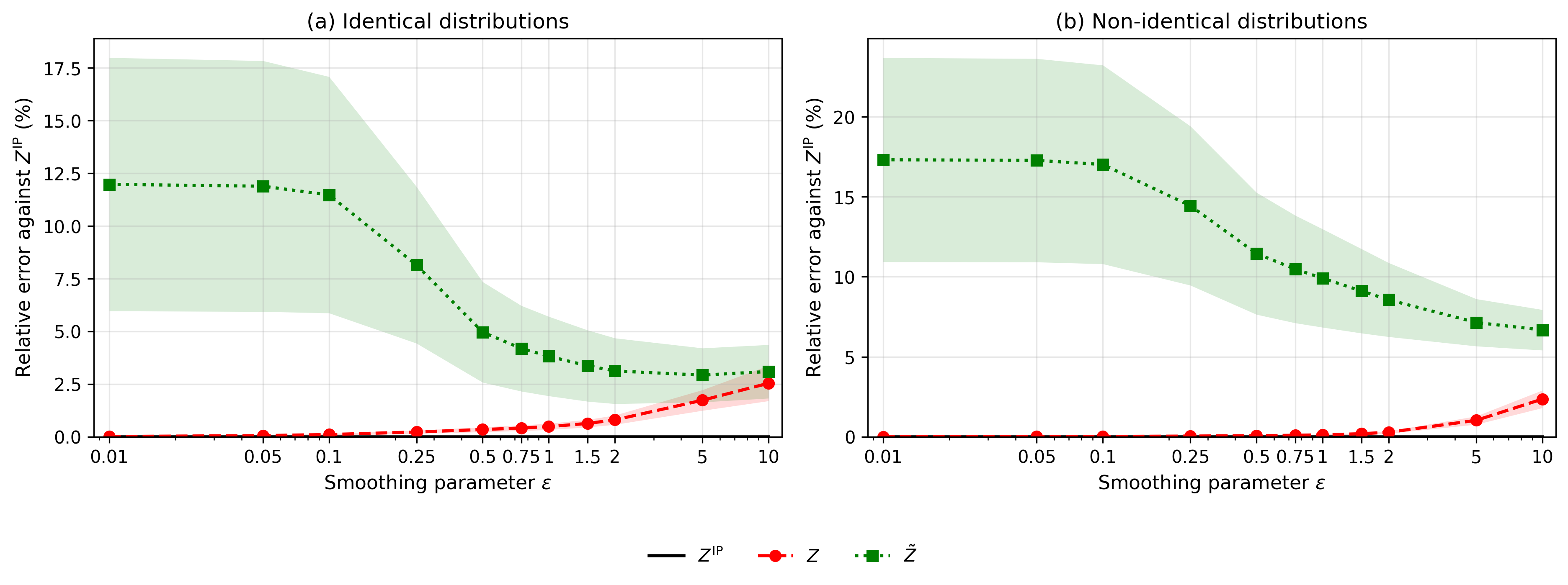}
    \caption{Relative error by different values of $\epsilon$. }
    \label{fig: error_eps}
\end{figure}

\begin{figure}
    \centering
    \includegraphics[width=0.9\linewidth]{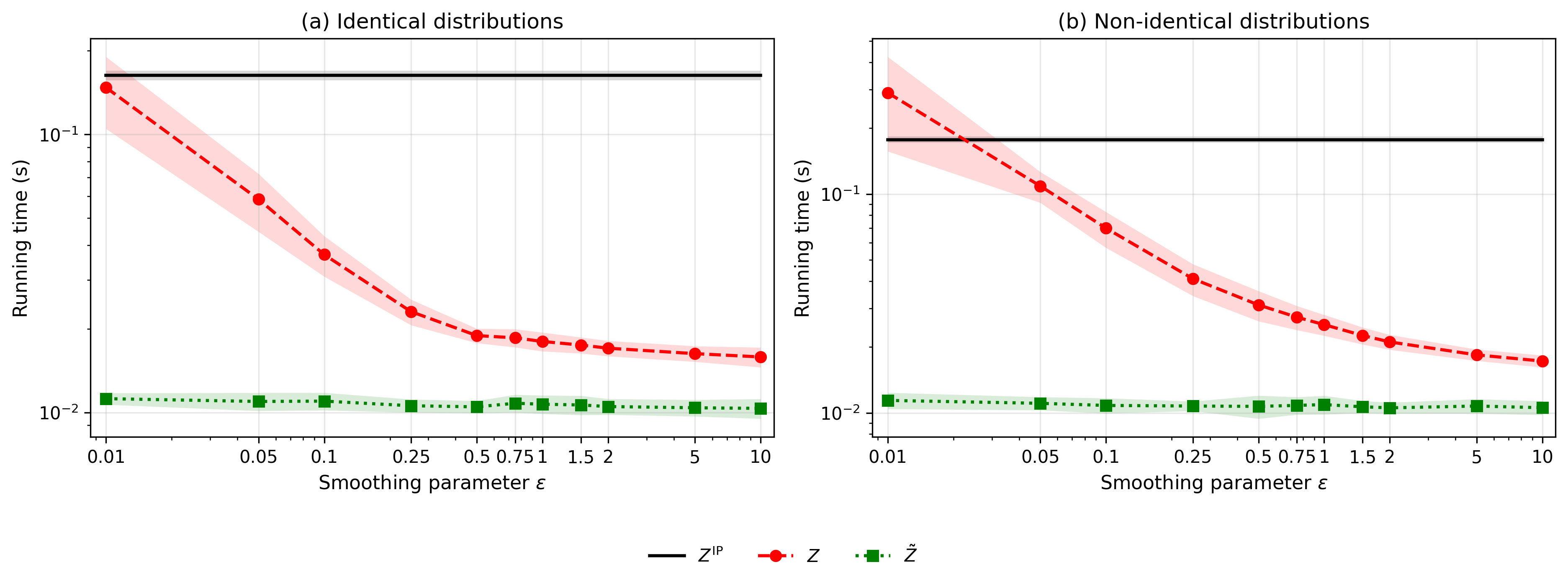}
    \caption{Running time by different values of $\epsilon$. }
    \label{fig: time_eps}
\end{figure}

In summary, one should choose the estimator with an appropriate $\epsilon$ based on the problem setting. When the network structure is simple, with a sufficiently large $\epsilon$, the electric approximation provides an efficient and relatively accurate estimate, especially when point distributions are nearly identical. When the network becomes more irregular and point distributions are drastically non-identical, the convex-flow formulation is more reliable and provides the most accurate estimations, and a moderate $\epsilon$ value can be used to accelerate the convergence of the Newton method. 
For very large-scale instances, when solving the full convex-flow problem becomes computationally expensive, the one-shot fast estimator provides a useful alternative. 

\subsection{Verification of random bipartite matching}
\label{sec: numerical_random}
In this section, we use the Sioux Falls network as an example and numerically validate the asymptotic properties derived for random bipartite matching. 

\subsubsection{Verification of scaling}
We first examine the scaling of the optimal matching cost as $n$
increases. 
The theoretical analysis predicts that $\Exp[Z]^\mathrm{IP} \asymp\sqrt{n}$ in identical distribution cases and $\Exp[Z^\mathrm{IP}] \asymp n$ in non-identical distribution cases. To verify them, we plot in Figure \ref{fig: scaling} the normalized quantity $\Exp[Z^\mathrm{IP}]/\sqrt{n}$ and $\Exp[Z^\mathrm{IP}]/n$, respectively,  
together with their 95\% confidence intervals. In the identical distribution case, $\Exp[Z^\mathrm{IP}]/\sqrt{n}$ remains approximately stable within a narrow range of $[19.9, 20.6]$ for the wider range of $n$ values. Similarly, in the non-identical distribution case, $\Exp[Z^\mathrm{IP}]/n$ stabilizes within $[8.475, 8.510]$. 
These numerical results directly support the derived $\sqrt{n}$- and $n$-scaling properties. 

\begin{figure}[htbp]
    \centering
\includegraphics[width=0.9\linewidth]{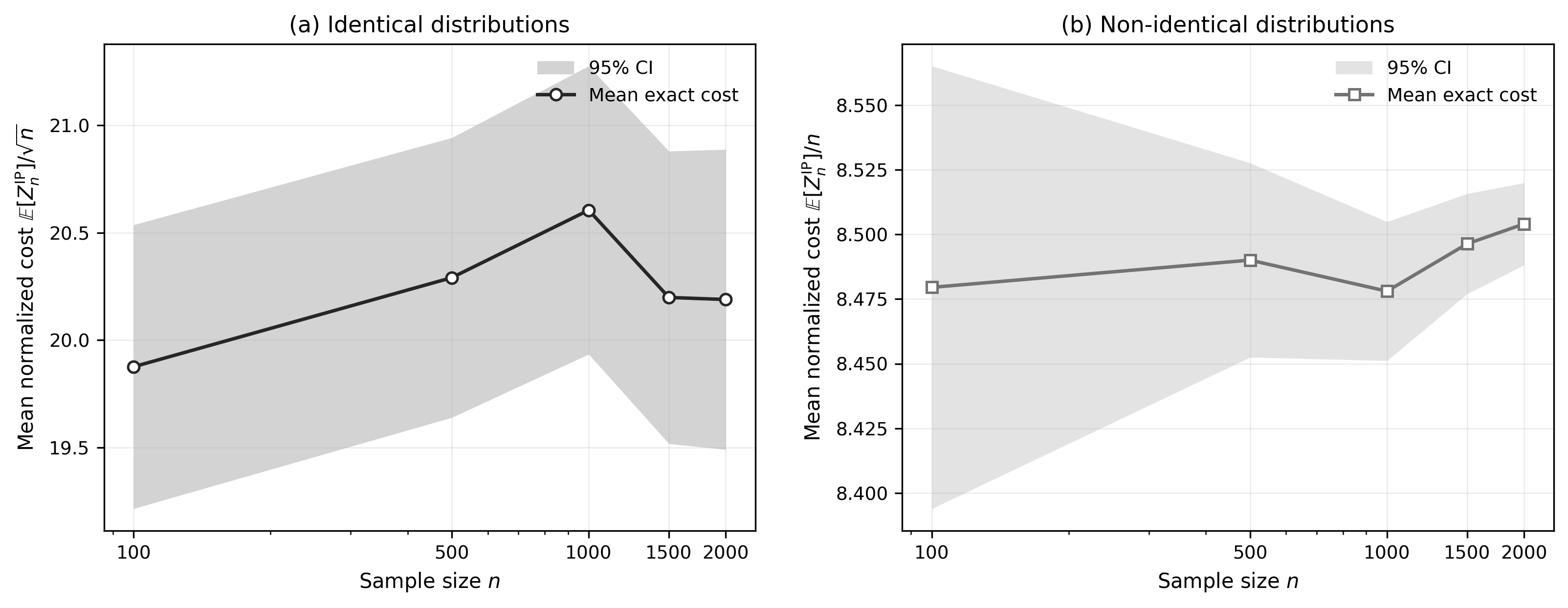}
    \caption{Verification of the scaling of normalized $\Exp[Z^\mathrm{IP}]$.}
    \label{fig: scaling}
\end{figure}

\subsubsection{Verification of sub-Gaussian flow}
We next examine the flow properties in the identical point distribution case. Figure \ref{fig: flow_interedge} shows the 
observed $c_e^\mathrm{int}$ values and fractions of inter-edge matches on randomly selected edges. In Figure \ref{fig: flow_interedge}(a), for each selected edge, the 100 values of $c_e^\mathrm{int}$ obtained from independently simulated instances are represented by black circle markers; their means are represented by red horizontal lines. It is observed that, for each edge, the distribution of $c_e^\mathrm{int}$ is approximately symmetric and bell-shaped, and the mean remains close to zero. These observations are consistent with the theoretical derivation that $c_e^\mathrm{int}$ is centered, symmetric, and sub-Gaussian. Figure \ref{fig: flow_interedge}(b) further shows the fraction of inter-edge matching among all matched pairs, which decreases steadily with $n$, indicating that matching becomes increasingly localized within individual edges. 

\begin{figure}[htbp]
    \centering
    \includegraphics[width=0.95\linewidth]{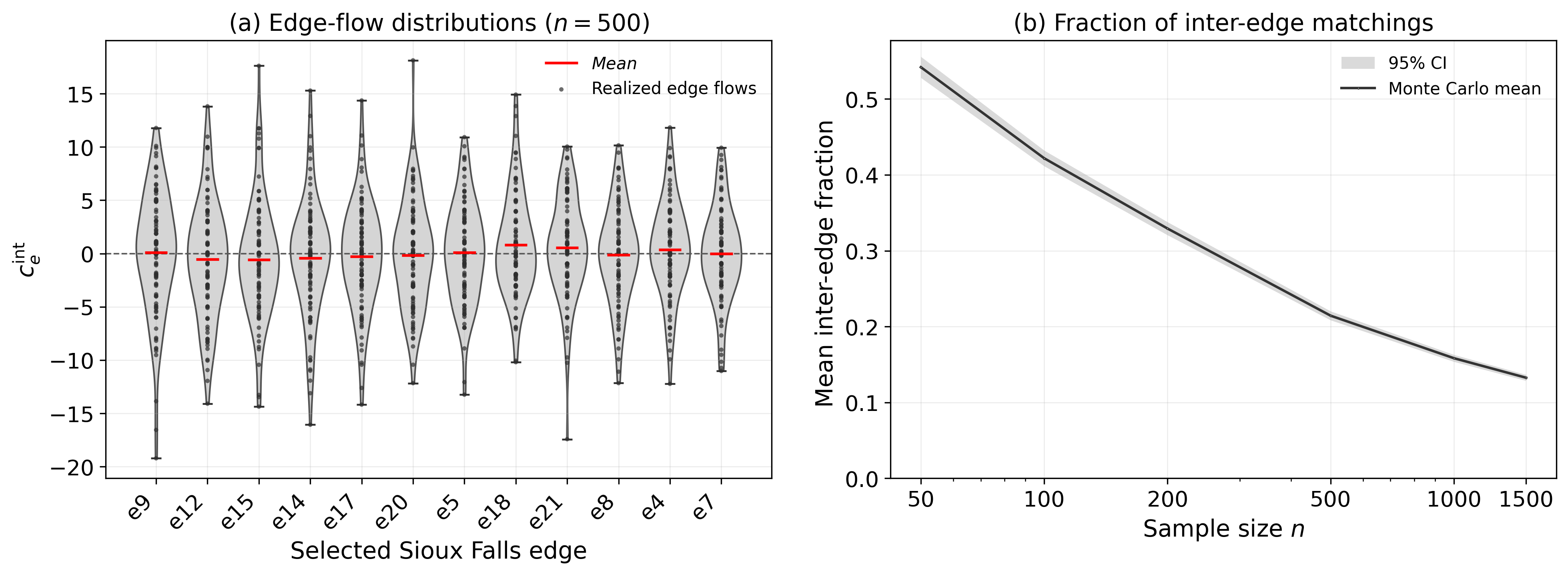}
    \caption{Verification of flow properties in identical distribution case. }
    \label{fig: flow_interedge}
\end{figure}

\subsubsection{Verification of asymptotic convergence}
Last we examine the convergence of $\Exp[Z^\mathrm{IP}]/n$ to $Z^\infty$ in the non-identical distribution cases, which is important for understanding the performance of the fast estimator $Z^\mathrm{1s}$. 
Figure \ref{fig: Zinf_conv} shows the comparison between $\Exp[Z^\mathrm{IP}]$ and $nZ^\infty$ for $\beta\in \{0.1, 1, 5, 10, 20\}$. 
The values of $\Exp[Z^\mathrm{IP}]$ and $nZ^\infty$ are represented by solid curves and hollow markers, respectively, while different colors correspond to different values of $\beta$. 
For all tested values of $\beta$, the difference between $\Exp[Z^\mathrm{IP}]$ and $nZ^\infty$ decreases with $n$, supporting the convergence of $\Exp[Z^\mathrm{IP}]/n$ to $Z^\infty$. However, the convergence is considerably slower when $\beta$ is small (i.e., when demand point distribution is more uniform). 
For larger values of $\beta$, $nZ^\infty$ closely approximates $\Exp[Z^\mathrm{IP}]$ over most of the tested $n$ values.

\begin{figure}
    \centering
    \includegraphics[width=0.95\linewidth]{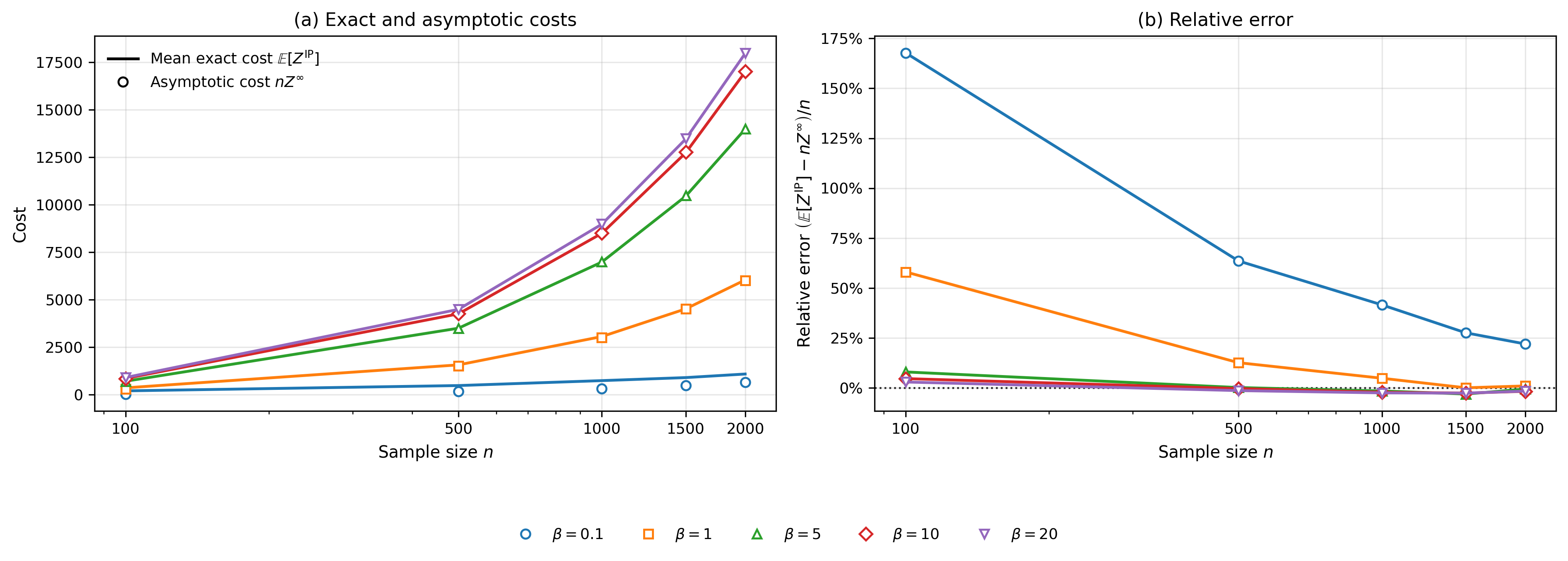}
    \caption{Verification of the convergence of $Z^\mathrm{IP}/n$ to $Z^\infty$.}
    \label{fig: Zinf_conv}
\end{figure}

\section{Conclusion}
\label{sec: conclusion}
This paper develops a flow model based framework for analyzing and approximating deterministic and random minimum-distance bipartite matching problems on general networks. 
First, the matching problem is reformulated as a separable convex-flow problem. 
By representing each matched pair as one unit of flow transporting through the network, the total bipartite matching cost $Z^\mathrm{IP}$ can be expressed as the sum of edgewise convex potential functions (determined by its cumulative supply-demand imbalance profile) that captures point matches within each edge, while network flow between edges captures point matches across edges. This reformulation replaces the full pairwise cost matrix with edgewise imbalance information and reveals how local matching costs and network-level matching jointly determine the optimal matching distance. We prove that the optimal objective value of the convex-flow formulation, $Z^\mathrm{CF}$, is exactly equal to $Z^\mathrm{IP}$.

Second, to enhance computational tractability and revealing interpretable properties, we introduce a smooth and monotone rearranged 
approximation of the edge imbalance profiles. The resulted approximate smooth convex-formulation, with optimal objective value $Z$, can be solved efficiently through its KKT conditions using standard methods (e.g., the Newton method). A first-order linearization of these conditions further yields an electric-network interpretation of the bipartite matching problem. Under this interpretation, the boundary matching flow corresponds to the electric current, the Lagrangian multipliers correspond to the node potential, and each edge has an endogenous resistance induced by the local slope of the rearranged imbalance profile. This connection leads to a resistance-based approximation with objective value $\tilde{Z}$ that requires only a weighted Laplacian solution and therefore avoids repeated nonlinear optimization. The problem is further extended into more general settings, including supply and demand points with heterogeneous masses, directed networks, and unbalanced matching problems with unequal number of supply and demand points.

Third, we study the random bipartite matching problem, where supply and demand points are independently sampled from two predetermined spatial distributions on the network. The analysis identifies two fundamentally different limiting regimes. When supply and demand points on all edges follow identical probability distributions, the expected optimal matching distance scales with $\sqrt{n}$, and the optimal boundary flows are centered, symmetric, and sub-Gaussian. This finding justifies an approximate distance prediction formula in an earlier work
\citep{zhai2026average}. As the number of points increases, matching becomes increasingly localized within each edge. When the supply and demand distributions differ in probability on at least some of the edges, a deterministic supply-demand imbalance dominates the stochastic fluctuation, and the total matching distance scales with $n$. In this regime, the normalized optimal flow asymptotically converges to a deterministic limiting flow, and the limiting edge resistances characterize how the distributional imbalance is ``cancelled'' through inter-edge matches across the network. These properties also support a one-shot approximate solution in which a limiting optimal network flow solution can be precomputed, and solution to each realized instance is a simple resistance-weighted projection from that limiting flow solution. The objective value induced by the projected flow is denoted by $Z^\mathrm{1s}$, which becomes increasingly accurate as the number of matched pairs grows.

A series of numerical experiments are conducted to verify the theoretical results in networks with different sizes and topological structures. The smooth convex objective $Z$ consistently provides the most accurate predictions and closely $Z^\mathrm{IP}$, with average relative errors below $1\%$ across all tested cases.
The resistance-based estimate $\widetilde{Z}$ achieves substantial computational savings and performs particularly well on simpler networks with identical point distributions. 
The one-shot estimator $Z^\mathrm{1s}$ provides the fastest online computation in the non-identical distribution regime and becomes increasingly accurate when the number of matches grows.
The numerical experiments also demonstrate the predicted trends of $\sqrt{n}$- and $n$- scaling, the concentration of the boundary flows, the asymptotically vanishing fraction of the inter-edge matches under the identical point distributions, and the convergence of the normalized matching cost to its deterministic limit under the non-identical point distributions.

This study has several limitations. All proposed formulations focus primarily on static matching problems and the matching cost is measured by the shortest-path distance. Future studies may extend the modeling framework to a more realistic scenario, including dynamic matching and more general cost functions (e.g., to account for factors such as traffic congestion and operation uncertainties). In addition, although the numerical experiments show that the resistance-based approximation performs particularly well on simple networks, the theoretical explanation for this behavior is not fully developed. Further analysis is needed to establish the convergence property or the error bound and clarify how the approximation accuracy depends on the network topology, problem scales, and imbalance in point distributions. 

\section{Acknowledgments}
The first author used OpenAI’s ChatGPT, a large language model, to assist with visualization code refinement and enhancement. The authors reviewed and verified all AI-assisted output and remain fully responsible for the accuracy, integrity, and content of the manuscript.
\pagebreak

\bibliographystyle{apalike}  
\bibliography{ref}  

\pagebreak
\section{Notation table}
\noindent
$G$: undirected network\\
$\mathcal{V}$: set of nodes in $G$\\
$\mathcal{E}$: set of edges in $G$\\
$e$: an edge in $\mathcal{E}$\\
$v$: a node in $\mathcal{V}$\\
$l_e$: length of edge $e$\\
$\mathbf{l}$: vector of edge lengths\\
$L$: total network length\\
$v_e^-,v_e^+$: two endpoint nodes of edge $e$\\
$\delta^-(v)$: set of outgoing edges from node $v$\\
$\delta^+(v)$: set of incoming edges to node $v$\\
$x$: local coordinate on an edge\\
$x_w$: coordinate of point $w$ on its edge\\
$(e,x)$: position on edge $e$ with local coordinate $x$\\
$\L_e(x)$: set of positions on edge $e$ between $0$ and $x$\\
$U$: set of supply points\\
$V$: set of demand points\\
$U_e$: set of supply points on edge $e$\\
$V_e$: set of demand points on edge $e$\\
$n$: number of supply points and demand points in the balanced case\\
$m$: number of demand points in the unbalanced case\\
$d(\cdot,\cdot)$: shortest-path distance on the network\\
$y_{uv}$: binary matching decision between supply point $u$ and demand point $v$\\
$Z^{\mathrm{IP}}$: optimal objective value of the integer-programming matching formulation\\
$f_e(x)$: net flow at position $(e,x)$\\
$S_e(x)$: cumulative supply-demand imbalance profile on edge $e$\\
$\mathbf{s}$: vector of edgewise total imbalances\\
$c_e$: boundary-flow shift on edge $e$\\
$\mathbf{c}$: vector of boundary-flow shifts\\
$\xi_e$: generic edge imbalance profile on edge $e$\\
$\Psi_e(c_e;\xi_e)$: edge potential induced by boundary value $c_e$ and profile $\xi_e$\\
$Z^{\mathrm{CF}}$: optimal objective value of the integer-valued convex-flow formulation\\
$Z^{\mathrm{RLX}}$: optimal objective value of the continuous relaxation of the convex-flow formulation\\
$\mathbf{I}^+$: incoming node-edge incidence matrix\\
$\mathbf{I}^-$: outgoing node-edge incidence matrix\\
$\mathbf{I}$: incidence matrix\\
$\mathbf{c}^{\mathrm{int}}$: optimal integer-valued boundary-flow solution\\
$\mathbf{c}^{\mathrm{rlx}}$: optimal solution of the relaxed convex-flow formulation\\
$\mathbf{c}^{\mathrm{rnd}}$: rounded flow vector obtained from the relaxed solution\\
$\mathbf{c}^{\mathrm{fes}}$: integral feasible flow vector after flow-conservation repair\\
$\Delta$: flow-conservation residual after rounding\\
$F_e(c;\xi_e)$: cumulative length function of profile $\xi_e$ on edge $e$\\
$F_e^-(c;\xi_e)$: left-limit cumulative length function\\
$Q_e(\cdot;\xi_e)$: generalized inverse of $F_e$\\
$Q_e(\cdot;S_e)$: monotone rearrangement of $S_e$\\
$Q_{e,\epsilon}$: smooth approximation of $Q_e(\cdot;S_e)$\\
$\epsilon$: smoothing parameter\\
$Z$: optimal objective value of the smooth approximate convex-flow formulation\\
$\mathbf{c}^*$: optimal solution of the smooth approximate convex-flow formulation\\
$\boldsymbol{\Lambda}$: Lagrangian multiplier vector for node conservation constraints\\
$\boldsymbol{\Lambda}^*$: optimal Lagrangian multiplier vector\\
$x_{0,e}$: reference point on edge $e$ for Taylor linearization\\
$\mathbf{x}_0$: vector of reference points\\
$\mathbf{c}_0(\mathbf{x}_0)$: baseline flow vector from the local uncoupled optimum\\
$\mathbf{R}(\mathbf{x}_0)$: diagonal resistance matrix at reference points $\mathbf{x}_0$\\
$\mathbf{L}(\mathbf{x}_0)$: weighted graph Laplacian induced by $\mathbf{R}(\mathbf{x}_0)$\\
$\mathbf{c}_0^*$: midpoint baseline flow vector, evaluated at $\mathbf{l}/2$\\
$\mathbf{R}^*$: midpoint resistance matrix\\
$\mathbf{L}^*$: midpoint weighted Laplacian matrix\\
$\tilde{\boldsymbol{\Lambda}}$: linearized approximation of the Lagrangian multiplier\\
$\tilde{\mathbf{c}}$: electric or resistance-based approximate flow\\
$\tilde{Z}$: objective value obtained by evaluating $\tilde{\mathbf{c}}$ in the original edge potential\\
$\|\cdot\|_{\mathbf{R}}$: resistance-weighted norm\\
$\mathbf{r}$: Taylor remainder vector in the electric approximation\\
$M$: number of partitions used to subdivide each edge\\
$k$: number of virtual absorption points per edge in the unbalanced approximation\\
$\mathcal{V}^k$: node set of the subdivided network induced by virtual points\\
$\mathcal{E}^k$: edge set of the subdivided network induced by virtual points\\
$\mathbf{a}$: absorbed-flow vector in the unbalanced formulation\\
$\mathbf{1}$: vector of ones with compatible dimension\\
$\mu$: probability measure for supply-point locations\\
$\nu$: probability measure for demand-point locations\\
$\mathcal{F}$: $\sigma$-algebra on the network sample space\\
$\zeta_e(x)$: limiting normalized imbalance profile on edge $e$\\
$\mathbf{s}^{\infty}$: limiting vector of terminal imbalance values\\
$Z^{\infty}$: limiting objective value of the normalized convex-flow problem\\
$\mathbf{c}^{\infty}$: optimal solution of the limiting convex-flow problem\\
$B$: standard Brownian bridge\\
$\mathbf{R}^{\infty}$: limiting resistance matrix\\
$\mathbf{L}^{\infty}$: limiting weighted Laplacian matrix\\
$\mathbf{c}^{\mathrm{1s}}$: one-shot asymptotic flow approximation\\
$Z^{\mathrm{1s}}$: objective value induced by the one-shot approximation\\
$n_e^{\mathrm{inter}}$: number of matches crossing edge boundary on edge $e$\\
$w_e$: center-weighted distance score assigned to edge $e$\\

\end{document}